\documentclass[preprint,12pt]{elsarticle}

\usepackage{amssymb}
\usepackage{amsmath}
\usepackage{amsthm}
\usepackage{url}

\newcommand{\comm}[1]{}

\newtheorem{teorema}{Theorem}[section]

\newtheorem{lema}[teorema]{Lemma}
\newtheorem{proposicion}[teorema]{Proposition}
\newtheorem{corolario}[teorema]{Corollary}

\newtheorem{hecho}[teorema]{Fact}

\newtheorem{particularizacion}[teorema]{Particularization}

\theoremstyle{definition}
\newtheorem{estrategia}[teorema]{Strategy}
\newtheorem{procedimiento}[teorema]{Procedure}

\newtheorem{conjetura}[teorema]{Conjecture}
\newtheorem{comentario}[teorema]{Comment}

\newtheorem{definicion}[teorema]{Definition}
\newtheorem{cuestion}[teorema]{Question}

\newtheorem{warning}[teorema]{Warning}

\newtheorem{convencion}[teorema]{Convention}

\newtheorem{ejemplo}[teorema]{Example}

\newtheorem{objetivo}[teorema]{Objective}

\newtheorem{observacion}[teorema]{Observation}
\newtheorem{remark}[teorema]{Remark}
\newtheorem{computacion}[teorema]{Computation}

\newtheorem{notacion}[teorema]{Notation}

\newtheorem{reminder}[teorema]{Reminder}

\DeclareMathOperator{\des}{des}
\DeclareMathOperator{\rcs}{rcs}
\DeclareMathOperator{\trun}{trun}
\DeclareMathOperator{\sym}{sym}
\DeclareMathOperator{\Ima}{Im}
\DeclareMathOperator{\Des}{Des}
\DeclareMathOperator{\suma}{sum}
\DeclareMathOperator{\un}{un}
\DeclareMathOperator{\mult}{mult}

\begin{document}

\begin{frontmatter}



\title{Spectrahedral relaxations of Eulerian rigidly convex sets}


\author{Alejandro González Nevado} 

\affiliation{organization={Universität Konstanz, Fachbereich Mathematik und Statistik},
            addressline={Universitätstraße 10}, 
            city={Konstanz},
            postcode={78464}, 
            state={Baden-Württemberg},
            country={Germany}}

\begin{abstract}
We study a generalization of Eulerian polynomials to the multivariate setting introduced by Brändén (and first discussed in \cite{branden2011proof,visontai2013stable, haglund2012stable}). Although initially these polynomials were introduced using the language of hyperbolic and stable polynomials, we manage to translate some restrictions of these polynomials to our real zero setting. Once we are in this setting, we focus our attention on the rigidly convex sets defined by these polynomials. In particular, we study the corresponding rigidly convex sets looking at spectrahedral relaxations constructed through the use of monic symmetric linear matrix polynomials of small size and depending polynomially (actually just cubically) on the coefficients of the corresponding polynomials. We analyze how good are the obtained spectrahedral approximations to these rigidly convex sets. We do this analysis by measuring the behavior along the diagonal, where we precisely recover the original univariate Eulerian polynomials. Thus we conclude that, measuring through the diagonal, our relaxation-based spectrahedral method for approximation of the rigidly convex sets defined by multivariate Eulerian polynomials is highly accurate. In particular, we see that this relaxation-based spectrahedral method for approximation of the rigidly convex sets defined by multivariate Eulerian polynomials provides bounds for the extreme roots of the corresponding univariate Eulerian polynomials that are better than these already found in the literature. All in all, this tells us that, at least close to the diagonal, the global outer approximation to the rigidly convex sets provided by this relaxation-based spectrahedral method is itself highly accurate.
\end{abstract}


\comm{
\begin{highlights}
\item Eulerian polynomials admit real stable, hyperbolic and real zero liftings
\item Construction of accurate approximations to the innermost ovaloid using a relaxation 
\item Global bounds in a lateral way
\item Global bounds, along the diagonal line, beat known bounds
\item Guessing eigenvectors through numeric experimentation increases accuracy in bounds
\end{highlights}
}

\comm{
\item Univariate Eulerian polynomials admit real stable, hyperbolic and real zero multivariate liftings
\item It is possible to construct accurate approximations to the innermost zero set (ovaloid) using a multivariate relaxation 
\item The relaxation provides global bounds in a lateral way
\item The global bounds provided by the relaxation are, along the diagonal line, the best in the current literature about univariate Eulerian polynomials
\item Successfully guessing approximations of eigenvectors using numeric experimentation provides venues and opportunities for more accuracy in the bounds established
}

\begin{keyword}
Eulerian polynomial \sep Descent set \sep Permutation \sep Hyperbolic polynomial \sep Stable polynomial \sep Real zero polynomial \sep Rigidily convex set \sep Spectrahedron



\end{keyword}

\end{frontmatter}



\section{Introduction}
\label{intro}

Several concepts will meet here in our way to relate Eulerian polynomials with objects common in the study of real algebraic geometry. Of course the main objects here are the well-known (univariate) Eulerian polynomials stemming from counting descents in permutations.

\begin{definicion}[Descents and univariate Eulerian polynomials]\label{primer}\cite[Sections 2.1 and 2.2]{haglund2012stable}
Let $\mathfrak{S}_{n}$ be the symmetric group considered as permutations on the ordered set $[n]:=\{1,\dots,n\}$ and written in its one-line form and consider $\mathfrak{S}=\bigcup_{n=0}^{\infty} \mathfrak{S}_{n}$ the union of all these symmetric groups $\mathfrak{S}_{n}$. Then we define the \textit{$n$-th univariate Eulerian polynomial} $$A_{n}(x)=\sum_{\sigma\in \mathfrak{S}_{n+1}}x^{\des(\sigma)}\in\mathbb{N}[x],$$ where \begin{gather*}
\des\colon\mathfrak{S}\to\mathbb{N},\sigma\mapsto\des(\sigma):=|\{i\mid\sigma_{i}>\sigma_{i+1}\}|.
\end{gather*} is the \textit{descent counting function}.
\end{definicion}

These polynomials verify some well-known nice properties that we collect here.

\begin{proposicion}
Univariate Eulerian polynomials are palindromic and all their roots are negative.
\end{proposicion}

The palindromicity helps us getting information for the smallest root looking at the largest root.

\begin{proposicion}
Let $p\in\mathbb{R}[x]$ be a palindromic polynomial. If $r$ is a root of $p$, so is $\frac{1}{r}$.
\end{proposicion}

We introduce now notation for the roots of the univariate Eulerian polynomials because these are the objects we will use to measure the accuracy of our bounds.

\begin{notacion}[Roots of Eulerian polynomials]
We write $A_{n}(x)=\prod_{i=1}^{n}(x-q_{i}^{(n)})$ with $q_{1}^{(n)}<\cdots<q_{n}^{(n)}$.
\end{notacion}

Thus the elements of $\mathfrak{S}_{n}$ are understood here as bijections from the ordered set $[n]\subseteq\mathbb{N}$ to itself so we can use its order to introduce descents. Additionally, we can easily see that the polynomials that we have introduced in the definition above are basically very special generating function. In particular, when we want to go from counting features within an object to forming polynomials, we can do this by transforming the information coming from these objects into functions. Usually the functions we obtain are either formal power series or, as in this case, polynomials, which are even better for their straightforward study. Precisely, we obtain polynomials as a consequence of the fact that we concentrate in finite objects and finite features to count with finitely many options.

\begin{definicion}[Generating polynomials]\cite[Section 2.2]{wilf2005generatingfunctionology}
The \textit{ordinary generating polynomial associated to the tuple} $(a_{1},\dots,a_{n})\in\mathbb{R}^{n}$ is the polynomial $p(x):=\sum_{i=1}^{n}a_{i}x^{i}.$ Building over this, if we consider a finite set $A$ of objects to be enumerated in some way through a \textit{measure function} $m\colon A\to\mathbb{N}$ we call the \textit{ordinary generating polynomial associated to the set $A$ through the measure function $m$} the polynomial $p(x):=\sum_{n>0}a_{i}x^{i}$, where $a_{i}:=|\{a\in A\mid m(a)=i\}|$ is the number of objects of size $i$ as measured by the measure $m$ defined over the set $A.$  
\end{definicion}

Note that, in general, the variable $x$ in a generating polynomial only serves as a placeholder for keeping track of the coefficients accompanying the corresponding monomials and does not fulfill in principle any additional role. This will change in the future when we introduce \textit{tags} for the variable in order to count finer features. In particular, these tags will allow our variables to carry further information about the objects measured and stored in these polynomial (multivariate) generating functions.

\begin{remark}
In this way, we can see the $n$-th univariate Eulerian polynomial as the generating polynomial for the descent statistic over the symmetric group $\mathfrak{S}_{n+1}$. In this case, $A:=\mathfrak{S}_{n+1}$ and the measure function on $A$ is given by the classic descent counting function $\des\colon\mathfrak{S}_{n+1}\to\mathbb{N},\sigma\mapsto\des(\sigma)$ introduced above.
\end{remark}

These polynomials have many properties that make them a very attractive and trending topic in current combinatorics. For instance, they are one of the prototypical examples of palindromic, unimodal, log-concave, real-rooted, gamma-positive polynomials stemming from combinatorics. These properties of polynomials are receiving more attention in connection to many different topics in combinatorics and algebraic geometry. And, thus, in this way, Eulerian polynomials become one of the main playgrounds in which we can perform experiments in order to understand better the relations and connections hidden between these properties and fields.

\begin{remark}[Eulerian variations]\label{variations}
In the journey of understanding better why these polynomials are so well behaved, other researchers have already looked into many directions. These explorations around Eulerian polynomials have produced many different variations of these polynomials exploiting techniques of finer-counting or extending the scope of the counting beyond simple permutations into related but more exotic objects. To know more about this and deeper connections with many other objects, we direct the interested reader to \cite{bona2022combin, petersen2015eulerian, kitaev2011patterns,katz2024matroidal} for various degrees of deeper and wider information about Eulerian polynomials, their generalizations, their properties and their coefficients.
\end{remark}

In order to understand the vast richness introduced in the works and directions mentioned in Remark \ref{variations} above, we observe that just one small piece of the many extensions mentioned in that remark gives as a result the main topic of this work. In particular, we will stay within the realm of usual permutations. But, additionally, we want to count over them in a finer way. For this, we just look with more attention at \textit{what} we are counting and \textit{where} we are counting.

\begin{estrategia}[Finer-counting]\label{finer}
When we are counting objects and its features in order to form elements in a ring of polynomials keeping track of our computations, a nice and natural way to lift our polynomials into more general rings is through the implementation of a finer way of counting that can distinguish better between different features or their appearance, allowing us to retrieve more information that we therefore need to store in a more sophisticated structure. This is how we can establish a direct and interesting parallel between jumping higher in the rings of polynomials collecting our combinatorial information and jumping higher in the sophistication of our counting methods and their granularity or attention to more and finer subtleties and details.
\end{estrategia}

To appreciate better the importance of this finer-counting see \cite{pemantle2008twenty} and, for a natural example of applications in combinatorics, see, e.g., \cite{mcsorley2009multivariate}. We will explore here many more advantages of this finer-counting strategy.

\begin{observacion}
We began this introduction defining a classical statistic over the symmetric groups defined as permutations of an ordered set. This paved the road for the introduction of the Eulerian polynomials. Now we observe, in the spirit of the key strategy above, that almost all the information about where descents happen and which elements are affected by them is lost just considering the map $\des$ introduced above.
\end{observacion}

The loss of information observed above happens because we are just taking the cardinal of the set $\{i\mid\sigma_{i}>\sigma_{i+1}\}$. This kills too much information that now we will prefer to keep registered in our polynomials. In particular, we want to register and keep the information of the (top) element in that set over which each descent happen.  For this reason, we have to think about ways of collecting more information in a meaningful way.

\begin{observacion}
[Tags and variables]
We can collect more information about the permutations by keeping track of the elements at which a descent happens. We will use these element as tags for new variables related to the original ones. Therefore, instead of having just one variable, we will have families of variables stemming from the original ones through the use of the elements at which a descent happens as tags. Thus, we will build more sophisticated multivariate polynomials, in the spirit of Strategy \ref{finer} above, carrying this extra information through these tags.
\end{observacion}

In order to tag variables in this way we will count features of the permutations in a finer way. This is how we will be able to collect more information about the permutations at play.

\begin{remark}[Counting finer]
In particular, we will tag using the tops of descents and ascents in a permutations. These will define two families of variables related though these tags. Thus we will have one family for the ascents and another for the descents. In order to introduce this extension of counting, we first need to refine the concepts introduced in Definition \ref{primer} used for counting descents in a very \textit{forgetful} way. We want to remember more.
\end{remark}

Ascents and descents appear naturally in sequences of elements in an ordered set. As we see permutations as ordered finite sequences of elements selected from an ordered set through the one-line notation used here, these features of ordered sequences appear as a very natural occurrence on these permutations. This is why they are also the natural place to look for in order to tag variables and extend the information retrieved by the polynomials.

\begin{definicion}[Tops]\label{defexcat}\cite[Definition 3.1]{visontai2013stable}
For a permutation $$\sigma=(\sigma_{1}\cdots\sigma_{n+1})\in\mathfrak{S}_{n+1}$$ with $\sigma_{i}\in[n+1]$ written in the one-line notation, we say that $\sigma_{i}\in[n+1]$ is a \textit{descent top} if $\sigma_{i}>\sigma_{i+1}.$ In a similar way, we say that $\sigma_{i+1}\in[n+1]$ is an \textit{ascent top} if $\sigma_{i}<\sigma_{i+1}.$ 
\end{definicion}

It is clear that, as $1=\min([n])$, we have that $1\in[n]$ can never be a top of anything. Now we know already the objects we will deal with. We will see how to pass from counting objects to forming \textit{interesting} polynomials. This analysis will also help us to realize why our finer-counting of descents (and ascents) produces several kinds of useful multivariate liftings of the original univariate polynomials.

This is the very general and broad topic of the initial sections of this article. But,  before we do that, we will need to introduce several concepts in order to stay self-contained. Therefore, for clarity, we explain the overall structure of this article in the next section.

\section{Contents}

In the Section \ref{intro} above we provided an introduction to the main strategy that we will exploit in order to put our polynomials in the correct setting: multivariate real zero polynomials. Here we explain the contents and the structures of the next sections.

First, in Section \ref{over}, we will give a short overview of the state of the art of the Eulerian polynomials that interest us so we can see where the main ideas giving rise to these polynomials come from. In Section \ref{preli} we will introduce the tools that we will use to study these polynomials. With these two sections we establish the main set-up of the tools used throughout this article.

In Section \ref{recurr} we explain how two seemingly unrelated approaches following different paths end up giving the multivariate liftings that we want for the Eulerian polynomials. In particular, we will see how one of these paths makes clear the nice polynomial properties of the multivariate polynomials while the other path exposes the nice combinatorial features they carry. These nice polynomial properties are central for us here and are therefore studied in more detail in Section \ref{rela}, where we expose the relations between these properties and finally establish the multivariate polynomials in the correct setting we need for our study.

Section \ref{count} deals with computations related to the application of the relaxation to the adequate polynomials that we establish and focus in Section \ref{rela}. In particular, we look at the combinatorics of the coefficients of real zero multivariate liftings because that is the setting required for applying the relaxation successfully.

Building the relaxation successfully is the step that allows us to venture into Section \ref{rec}, where we study previous bounds already appearing in the literature about Eulerian polynomials. With this information about these previously known bounds, we can prove in Section \ref{ext} that the best known bounds can actually be recovered through the use of the relaxation. In that section, we also explain in general how we can retrieve bounds from the relaxation using different techniques and approaches.

We will also have a short trip around the asymptotic estimations of the roots of Eulerian polynomials in order to understand better the recent interest in them. We do this in Section \ref{asy} and it will help us comparing and understanding further how good is the relaxation with respect to the bounds already in the literature in asymptotic terms.

This study allows us to find the best univariate bound (that is, the best bound before we use multivariate liftings) in Section \ref{bestuni} using the determinant of the relaxation. There we also see how we can extract another bound with the same first asymptotic term using linearization though a suitable vector to construct an approximation to that optimal bound. This technique allows us to see that we can compute very accurate bounds through adequate linearizations. This observation will be key in multivariate approaches (after we lift our univariate polynomials to the multivariate setting), where computing or dealing with the corresponding determinants of the relaxation will be highly unpractical.

In Section \ref{previous} we see that the univariate relaxation has appeared before in the literature and study in depth this appearances. This study let us conclude that, if we want to to obtain more information about our polynomials, we must find the way to exploit the whole power of the relaxation through adequate multivariate liftings. This is so because we see that the univariate relaxation already equals the best bounds in the literature and want improve these. The natural way to achieve this improvement is by increasing the size of the LMPs generated in the process of constructing the relaxation. This observation clearly leads us towards the way of lifting to the multivariate case.

In this multivariate environment we find after lifting our univariate polynomials to multivariate ones, we will see that we face the problem that these previous bounds are given for the univariate polynomials while the bounds provided by the relaxation are global and multivariate. Therefore, as a consequence, we will need to establish, in Section \ref{ame}, a way to measure naturally the accuracy of the relaxation by what it does near the line in which the univariate (real-rooted) Eulerian polynomials are injected within our multivariate (real zero) polynomials (liftings).

Once the points above are clear, we will center our effort into working with the relaxation in order to refine our results as much as possible in Section \ref{wo}. That analysis will put us under the search of correct vectors as features to work effectively with the relaxation. In Section \ref{ch} we will see the dangers of not putting enough attention to the choice of vectors we make while in Section \ref{multrel} we will describe the way to make sure that we are appropriately choosing our simplifying vectors. This required attention to the choice of vectors and the description of the strategies to make sure that we are appropriately choosing them is related to the fact that we have to ensure that, through the structure of these choices of simplifying (linearizing) vectors, we use the whole power of the multivariate relaxation. We have to do this to guarantee that we are not throwing away the advantage gained through going multivariate through naive and unhelpful choices that ultimately kill that edge we initially gained making the jump to the multivariate setting.

After we analyze the best form for the choice of linearizing vectors, we study in Section \ref{limitations} the limitations of any possible improvement that these linearizing vectors are going to provide to our measures of accuracy. After this, we prepare in Section \ref{preparation} for dealing with the task of extracting and confirming an optimal bound from the linearization of the relaxation correctly chosen before. We determine in Section \ref{formof} the form of this new optimal (and uniparametric) bound extracted from the multivariate relaxation. When we have this formal bound, we optimize over its unique parameter in Section \ref{optimization} in order to obtain the best possible bound among these with the given form. During this optimization process, radicals will appear in our expressions. In order to correctly deal with expressions with radicals while we perform asymptotic comparisons, we have to study how we should \textit{correctly} deal with our expressions in the presence of radicals. We do this in Section \ref{correct}. This correct management of radicals lead us finally to the correct way of performing the comparison between the new multivariate bounds and the previous univariate ones. We do this fundamental comparison in Section \ref{finalcom}. This journey ends up when, in Section \ref{multapr}, we see how these adequate choices of vectors end up providing us with nice and highly accurate global bounds when they are measured through the restriction to the diagonal.

Finally, we devote the last two sections to describe an overview of the discoveries we realized in this article, how these findings develop and why the analysis made here advances new results. In Section \ref{alook} we provide an analysis of all the possible paths that the observation and results obtained here open from future research. We set, at the same time, several objectives for the future in these directions.

In particular, we conclude in Section \ref{con} commenting on the next most immediate objective that directly connects with the explorations made in this paper. This objective points towards the possibility of obtaining exponentially better asymptotic bounds through the choice of better sequences of vectors having structures allowing us to retrieve further information about the multivariate relaxations (and therefore about the multivariate polynomial lifitings of our initial univariate Eulerian polynomials) when we linearize our problem of establishing bounds for the extreme roots of the Eulerian polynomials through them. In this way, we can set up and plan our journey through the many paths that our findings here open. Thus we can reasonably advance, decide and justifiably guess what are the most natural and promising steps to follow next in the directions explored and introduced here.

\section{Overview}
\label{over}

The process that we follow here to form these interesting combinatorial polynomials through generating sequences is clear. We start with a set of objects and construct its generating polynomial in the way described in Section \ref{intro}. When we have this generating polynomial, it becomes clear that performing a finer counting over these initial objects and adding the new information into our storage through the use of a suitable tagging mechanism for variables allows us to carry and store more information within our polynomial. In particular, the multivariate monomials formed by the new mechanism of tagging variables will let us to produce a natural lifting of the univariate polynomial we started with. In this way, we construct the multivariate liftings of the initial polynomial. In particular, we are interested in Eulerian polynomials and their multivariate generalizations, which are just a case of this construction in combinatorics. These polynomials emerge naturally when we count descents (and their tops) in permutations happening (or defined) over ordered sets.

We introduced already the univariate Eulerian polynomials. We did this because we will see in Section \ref{ame} that a measure of our success in dealing with the global approximation provided by the relaxation will rest on how well we can approximate the extreme roots of these univariate polynomials. We will explain that better in the mentioned section. These measures of success when we try to asymptotically approximate sets of zeros of multivariate polynomials could benefit also the approaches taken in \cite{boese2003accurate, henrion2023polynomial, melczer2021multivariate,mourrain2000multivariate} and, preeminently (for considering the diagonal method), \cite{pemantle2008twenty} for problems seemingly related to the one we face here.

In general, we will use the univariate versions of our Eulerian polynomials and what is already known about them as convenient devices of measure of success in (asymptotic) accuracy during our explorations into the multivariate setting. In particular, we will see that the simpler and better understood univariate polynomials will constitute the perfect tool to measure and check how much a given technique, method or approach improves the global bound of the rigidly convex set that we can construct in each case. We decide to do this precisely because of how much is already known about univariate versions in contrast to how little is (widely) known about the multivariate ones. The recurrent definition of these polynomials is one of these useful (widely) well-known facts.

\begin{proposicion}[Recurrent definition of Eulerian polynomials]\label{recurrenceEuler}
Setting the initial polynomial $A_{1}(x)=1+x$ we have, for $n>1$, that $$A_{n}(x)=x(1-x)A'_{n-1}(x)+(1+nx)A_{n-1}.$$
\end{proposicion}

This recurrent definition will prove very useful in the future. Coming back to the combinatorial definition, how we lift these polynomials to another having more variables is by keeping track of the elements at which a descent happens, as can be seen in \cite{haglund2012stable, visontai2013stable}. This additional information is thus encoded in the form of a multivariate polynomial (which is, in fact, just a finite multivariate generating series).

This is the main idea that will put us in the reach of the next tool: the relaxation. We need to introduce some terminology and results that will help us to understand this relaxation and its properties better.

\section{Preliminaries}
\label{preli}

We will use a relaxation introduced in \cite{main}. We will present here the main concepts from this reference in order to keep our exposition as self-contained as possible. We will not include proofs in this section, as these can be found in the cited reference. We remind to the reader that all the matrices we will consider here are real symmetric.

Before introducing the relaxation, we will set the environment in which this relaxation works. Thus we will focus our attention around the following generalization of being real-rooted for multivariate polynomials. These are the polynomials we are seeking to find and the ones over which our methods work because for them we have, in fact, \textit{a relaxation}, as we will see.

\begin{definicion}[Real zero polynomial]\cite[Definition 2.1]{main}\label{RZ}
Let $p\in\mathbb{R}[\mathbf{x}]$ be a polynomial. We say that $p$ is a \textit{real zero polynomial} (or a \textit{RZ polynomial} for short) if, for all directions $a\in\mathbb{R}^{n}$, the univariate restriction $p(ta)$ verifies that all its roots are real, i.e., if, for all $\lambda\in\mathbb{C}$, we have $$p(\lambda a)=0 \Rightarrow \lambda\in\mathbb{R}.$$
\end{definicion}

Notice that a RZ polynomial has to verify $p(0)\neq0$, although there are extensions of this that we do not deal with here, see, e.g., \cite{knese2016determinantal}. Thanks to this observation we can introduce the following subset of its domain $\mathbb{R}^{n}$. This subset will be central to our research.

\begin{definicion}[Rigidly convex set]\cite[Definition 2.11]{main}\label{RCS}
Let $p\in\mathbb{R}[\mathbf{x}]$ be a RZ polynomial and call $V\subseteq\mathbb{R}^{n}$ its set of zeros. We call the \textit{rigidly convex set} of $p$ (or its \textit{RCS} for short) the Euclidean closure of the connected component of the origin $\mathbf{0}$ in the set $\mathbb{R}^{n}\smallsetminus V$ and denote it $\rcs(p)$.
\end{definicion}

We can give an equivalent definition. Both appear commonly in the literature.

\begin{proposicion}[Equivalent definition]
Let $p\in\mathbb{R}[\mathbf{x}]$ be a RZ polynomial. Then $$\rcs(p):=\{a\in\mathbb{R}^{n}\mid\forall\lambda\in[0,1)\ p(\lambda a)\neq0\}.$$ 
\end{proposicion}

\begin{proof}
Straightforward.
\end{proof}

There are many open questions in the study of RCSs, see, e.g., \cite{jorgens2018hyperbolicity,lourencco2024hyperbolicity,ito2023automorphisms,oliveira2020conditional} having in mind the direct relation between hyperbolic and RZ polynomials that can be found in \cite{main}. A remarkable one is the generalized Lax conjecture, which stem from the fact that a wide source of examples of RZ polynomials are these of the form $p=\det(I_{d}+x_{1}A_{1}+\cdots+x_{n}A_{n})$ with $A_{1},\dots,A_{n}\in\sym_{d}(\mathbb{R})$ for some $d\in\mathbb{N}$. See \cite{helton2007linear,vinnikov2012lmi} for more information about this conjecture and related ones already solved.

\begin{conjetura}[Generalized Lax conjecture, GLC]\cite[Conjecture 8.1, Theorem 8.8]{main}
Let $p\in\mathbb{R}[\mathbf{x}]$ be a RZ polynomial with $p(0)=1$. Then there exist another RZ polynomial $q\in\mathbb{R}[\mathbf{x}]$ and symmetric matrices $A_{1},\dots,A_{n}\in\sym_{d}(\mathbb{R})$ for some $d\in\mathbb{N}$ such that \begin{enumerate}
\item $pq=\det(I_{d}+x_{1}A_{1}+\cdots+x_{n}A_{n})$
\item $\rcs(p)\subseteq\rcs(q).$
\end{enumerate}
\end{conjetura}

This conjecture is only resolved in certain special cases. The conjecture is central in the real algebraic theory of these polynomials and also in optimization, but we will not develop here further details about the conjecture and its status because it is not important for the objective of this article.

The main important consequence of this conjecture is the fact that it has stimulated the development of tools, techniques and method suitable to understand better these RCSs. One of these methods is the relaxation introduced in \cite{main} that we will discuss here.

The idea of studying these sets with a \textit{relaxation} comes from the field of mathematical optimization. There, a relaxation is a well-behaved approximation to a feasibility set in the following sense. For a broader explanation about the concept of relaxation in mathematical optimization see, e.g., \cite[Chapter 12]{du2001mathematical}.

\begin{definicion}[General concept of relaxation]
Consider the minimization problem $z=\min\{c(\mathbf{x})\mid \mathbf{x}\in X\subseteq\mathbb{R}^{n}\}$. A \textit{relaxation} of that problem is another minimization problem of the form $z=\min\{\Tilde{c}(\mathbf{x})\mid \mathbf{x}\in \Tilde{X}\subseteq\mathbb{R}^{n}\}$ with $\Tilde{X}\supseteq X$ and $\Tilde{c}(\mathbf{x})\leq c(\mathbf{x})$ for all $\mathbf{x}\in X.$
\end{definicion}

RCSs are treated in this literature as feasible sets for convex optimization problems. Hence we obtain the name of our tool from this connection to optimization. As, in principle, we only care about the sets, what we obtain is an outer approximation.

\begin{definicion}[Relaxation of a RCSs]
Given the RCS $0\in C\subseteq \mathbb{R}^{n}$, we call $\Tilde{C}$ a relaxation of $C$ simply if $C\subseteq\Tilde{C}.$
\end{definicion}

Thus we see that, when viewed from the origin, a relaxation of a RCS is just an approximation of that set whose Euclidean (topological) border lies out of the Euclidean (topological) interior of the component of the origin cut out by $C$. Relaxations are interesting because sometimes these sets have a structure that is easier to understand or deal with in algorithmic optimization. In our case, we will want these relaxations to be \textit{spectrahedral}, which means that they should admit the following representation.

\begin{definicion}[Spectrahedra]\cite{ramana1995some}
A set $S\subseteq\mathbb{R}^{n}$ is a \textit{spectrahedron} if it can be written as $$S=\{a\in\mathbb{R}^{n}\mid A_{0}+\sum_{i=1}^{n}a_{i}A_{i} \mbox{\ is PSD\ }\}$$ with $A_{i}$ real symmetric matrices of the same size. We call $A_{0}+\sum_{i=1}^{n}x_{i}A_{i}$ a \textit{linear matrix polynomial} (or LMP) and the condition $$a\in\mathbb{R}^{n} \mbox{\ with \ } A_{0}+\sum_{i=1}^{n}a_{i}A_{i} \mbox{\ is PSD\ }$$ a \textit{linear matrix inequality} (or LMI) on $a\in\mathbb{R}^{n}$.
\end{definicion}

We also want our relaxation to be relatively easy to compute. The functions used to compute the relaxation stem from the logarithm of a polynomial that we can easily define as follows.

\begin{definicion}[Logarithm for power series]\cite[Definition 3.1, Definition 3.2]{main}
Let $p=\sum_{\alpha\in\mathbb{N}_{0}^{n}}a_{\alpha}\mathbf{x}^{\alpha}\in\mathbb{R}[[\mathbf{x}]]$ be a real power series. We define $$\log{p}:=\sum_{k=1}^{\infty}(-1)^{k+1}\frac{(p-1)^{k}}{k}\in\mathbb{R}[[\mathbf{x}]]$$ the \textit{logarithm} of $p$.
\end{definicion}

With the logarithm, we can now discuss the linear maps that will form the backbone of the construction of the relaxation introduced in \cite{main} entrywise. In order to define these linear maps, we look again at the power series as follows.

\begin{definicion}[$L$-form]\cite[Proposition 3.4]{main}
Let $p\in\mathbb{R}[\mathbf{x}]$ be a polynomial with $p(0)\neq0$. We define the linear form $L_{p}$ on $\mathbb{R}[\mathbf{x}]$ by specifying it on the monomial basis of $\mathbb{R}[\mathbf{x}]$. We set $L_{p}(1)=\deg(p)$ and define implicitly the rest of the values by requiring the identity of the formal power series $$-\log\frac{p(-x)}{p(0)}=\sum_{\alpha\in\mathbb{N}_{0}^{n}\\\alpha\neq0}\frac{1}{|\alpha|}\binom{|\alpha|}{\alpha}L_{p}(\mathbf{x}^{\alpha})\mathbf{x}^{\alpha}\in\mathbb{R}[[\mathbf{x}]].$$
\end{definicion}

These linear maps verify several interesting identities coming from the fact that they are defined through the use of a logarithm. The most important values of these linear forms for us will be those reached on monomials of degree up to $3$. These values are easy to compute.

\begin{computacion}[$L$-form degree up to three]\cite[Example 3.4]{main}
Suppose $p\in\mathbb{R}[\mathbf{x}]$ truncates to $$\trun_{3}(p)=1+\sum_{i\in[n]}a_{i}x_{i}+\sum_{i,j\in[n]\\i\leq j}a_{ij}x_{i}x_{j}+\sum_{i,j,k\in[n]\\i\leq j\leq k}a_{ijk}x_{i}x_{j}x_{k}.$$ Then we have that \begin{gather*}L_{p}(x_{i})=a_{i},\\L_{p}(x_{i}^{2})=-2a_{ii}+a_{i}^{2},\\L_{p}(x_{i}^{3})=3a_{iii}-3a_{i}a_{ii}+a_{i}^{3}\end{gather*} for all $i\in[n]$, while \begin{gather*}L_{p}(x_{i}x_{j})=-a_{ij}+a_{i}a_{j},\\L_{p}(x_{i}^{2}x_{j})=a_{iij}+a_{i}a_{ij}-a_{j}a_{ii}+a_{i}^{2}a_{j}\end{gather*} for all $i,j\in[n]$ with $i<j$ and $$L_{p}(x_{i}x_{j}x_{k})=\frac{1}{2}(a_{ijk}-a_{i}a_{jk}-a_{j}a_{ik}-a_{k}a_{ij}+2a_{i}a_{j}a_{k})$$ for all $i,j,k\in[n]$ with $i<j<k$.
\end{computacion}

We stop at degree $3$ for a very practical reason: the relaxation introduced in \cite{main} will only use these values. In particular, the relaxation has the form of a LMP whose initial matrix is PSD and whose entries are defined in terms of values of the $L$-form on monomials of degree up to three. For clarity on building this, we will use mold matrices and entrywise application of maps.

\begin{notacion}[Mold matrices and entrywise transformations]
We denote $L\circledcirc M$ the tensor in $B^{s_{1}\times\cdots\times s_{k}}$ obtained after applying the map $L\colon A\to B$ to each entry of $M\in A^{s_{1}\times\cdots\times s_{k}}.$ The tensor $M$ is called the \textit{mold tensor} and the map $L$ is the \textit{molding map}.
\end{notacion}

In particular, we construct the relaxation using a very simple mold matrix that contains all the variables: the symmetric moment matrix indexed by all the monomials in the variables $\mathbf{x}$ up to degree $1$. Due to its relevance in the construction of the relaxation, we give it a name.

\begin{notacion}
We denote $M_{n,\leq1}:=(1, x_{1}, \cdots, x_{n})^{T}(1, x_{1}, \cdots, x_{n})=$ \begin{equation}\label{mainmoment}
\begin{pmatrix}
1 & x_{1} & \cdots & x_{n}\\
x_{1} & x_{1}^{2} & \cdots & x_{1}x_{n}\\
\vdots & \vdots & \ddots & \vdots\\
x_{n} & x_{1}x_{n} & \cdots & x_{n}^2
\end{pmatrix}
\end{equation} the \textit{mold matrix indexed by monomials of degree up to $1$}.
\end{notacion}

We build the relaxation using this simple mold matrix and the molding map given by the $L$-form associated to the polynomial we want to study. This is how we construct the relaxation.

\begin{definicion}[Relaxation]\cite[Definition 3.19]{main}
Let $p\in\mathbb{R}[\mathbf{x}]$ be a polynomial with $p(0)\neq0$ and consider the symmetric matrices $A_{0}=L_{p}\circledcirc M_{n,\leq1}$ and $A_{i}=L_{p}\circledcirc (x_{i}M_{n,\leq1})$ for all $i\in[n]$. We call the linear matrix polynomial $$M_{p}:=A_{0}+\sum_{i=1}^{n}x_{i}A_{i}$$ the \textit{pencil associated to $p$} and $$S(p):=\{a\in\mathbb{R}^{n}\mid M_{p}(a) \mbox{\ is PSD}\}$$ the \textit{spectrahedron associated to $p$}.
\end{definicion}

Now that we know the process of construction of the object, we can give the main result that makes this object interesting for us. The relaxation is indeed so.

\begin{teorema}[Relaxation]\cite[Theorem 3.35]{main}
Let $p\in\mathbb{R}[\mathbf{x}]$ be a RZ polynomial. Then $\rcs(p)\subseteq S(p).$
\end{teorema}

Once we have established that our approximation is indeed a relaxation when we apply it to RZ polynomials, we have to delve into the polynomials that we want to investigate in order to realize that we can apply this construction to some instances of these polynomials. We will see how several generalizations of univariate Eulerian polynomials to the multivariate setting provide us with multivariate liftings of univariate polynomials that are \textit{close enough} to be RZ. This is the next step in our exploration. We develop this in the next section, where we center more clearly around Eulerian polynomials and how to construct their generalizations and liftings with a view towards obtaining RZ polynomials over which we can apply the tool given by the relaxation constructed and discussed here.

\section{Recurrence and finer-counting as equivalent paths towards multivariability}\label{recurr}

From the combinatorics perspective, the most natural way to introduce multivariate Eulerian polynomials is through finer-counting in the descents. However, if we want to see that these polynomials behave well in the sense of the next definition, we need to construct them through a nice enough recursivity consisting of a \textit{stability preserver}.

\begin{definicion}[Stability and stability preservers]\cite[Subsection 2.5]{haglund2012stable} and \cite[Definition 2.5]{visontai2013stable}
Let $p\in\mathbb{C}[\mathbf{x}]$ be a polynomial and define the set $$\mathcal{H}:=\{z\in\mathbb{C}\mid\Ima(z)>0\}.$$ We say that $p$ is \textit{stable} if $p\equiv0$ or $p(x)\neq0$ whenever $x\in\mathcal{H}^{n}$. We call $p$ \textit{real stable} when it is stable and has only real coefficients so $p\in\mathbb{R}[\mathbf{x}]$. Moreover, we say that a map $f\colon A\to B$ between the set of polynomials $A,B\subseteq\mathbb{C}[\mathbf{x}]$ is a \textit{(real) stability preserver} if $f(p)\in B$ is (real) stable whenever $p\in A$ is (real) stable.
\end{definicion}

We will be specially interested in linear stability preservers given in the form of finite order linear partial differential operators with polynomial coefficients on $\mathbb{C}[\mathbf{x}]$ that preserve stability. Our reason for looking at these particular operators is that we possess a satisfactory classification of them. This makes them easier to understand and work with. Using these tools, it is possible to prove that the following polynomials are stable thanks to the recursion that defines them.

\begin{definicion}\cite[Theorem 3.2]{haglund2012stable}
and \cite[Theorem 3.3]{visontai2013stable}\label{multieulerian}
Let $$A_{n}(\mathbf{x},\mathbf{y}):=\sum_{\sigma\in\mathfrak{S}_{n+1}}\prod_{i\in\mathcal{DT}(\sigma)}x_{i}\prod_{j\in\mathcal{AT}(\sigma)}y_{j},$$ where $$\mathcal{DT}(\sigma):=\{\sigma_{i}\in[n+1]\mid i\in[n],\sigma_{i}>\sigma_{i+1}\}\subseteq[2,n+1]$$ is the \textit{descent top set of $\sigma$} and $$\mathcal{AT}(\sigma):=\{\sigma_{i}\in[n+1]\mid i\in[n],\sigma_{i}<\sigma_{i+1}\}\subseteq[2,n+1]$$ is the \textit{ascent top set of $\sigma$}, be the \textit{descents-ascents $n$-th Eulerian polynomial}.
\end{definicion}

It is clear that there are indices at which we cannot possibly have a descent (or an ascent). In the same way, there are elements that can never be a top. In fact, $1$ can never be a top. This is so because $1\in[n+1]$ is smaller than any other element. Moreover, we cannot have a descent (or an ascent) at index $n+1$ because nothing follows it in the one-line notation. Thus, as we use tops to tag variables, the variable tagged by the element $1$ will never actually appear. This phenomenon implies that the variable tagged by $1$ becomes a ghost variable. This variable just exists because we are tagging by a set that needs a minimum that can therefore never be a top. For this reason, this element will never actually appear as a tag of a variable. It is important to remember this. In the future, not having this in mind could cause misunderstandings. These polynomials verify what we mentioned. To know more about closely related concepts of ghost variables, we direct the reader to \cite{unruh2019quantum,chevalier2020sharing}, but we warn that we use here a much lighter version of that concept.

\begin{teorema}[Real stability of Eulerian polynomials]\cite[Theorem 3.2]{haglund2012stable}
and \cite[Theorem 3.3]{visontai2013stable}
The polynomial $A_{n}(\mathbf{x},\mathbf{y})\in\mathbb{R}[\mathbf{x},\mathbf{y}]$ is real stable for all $n\in\mathbb{N}.$ 
\end{teorema}

As we talked about ghost variables we better notice that they are here with us. In fact, observe that we always have that $1\notin \mathcal{DT}(\sigma)$ and therefore $A_{n}$ has actually only the $n$ variables $x_{2},\dots x_{n+1}$.

Now that we finally have a nice sequence of multivariate polynomials, note that this sequence is nice in the sense that it lifts to the multivariate setting the property of being real-rooted. We can look at the many other generalizations of being real-rooted that there exist in the literature. In the end, we want to see how these relate in our particular setting. Remember that we are looking for RZ polynomials that we can feed into the relaxation introduced above.

\begin{remark}[Other liftings of real-rootedness]
We saw already that RZ polynomials constitute a natural generalization of being real-rooted. Polynomial hyperbolicity \cite[Definition 6.1]{main} and polynomial stability \cite[Subsection 2.5]{haglund2012stable} and \cite[Definition 2.5]{visontai2013stable} are other related generalizations of this concept.
\end{remark}

Now we have to study and compare how these related liftings interact. The connection between these generalizations of real-rootedness is central to our application of the relaxation as a tool. For this reason, we devote the next section to this.

\section{Relations between stability, hyperbolicity and real-zeroness}\label{rela}

We need RZ polynomials to feed into the relaxation. However, in the Section \ref{recurr} above we obtained real stable polynomials instead. The answer for our struggle is that these families of polynomials are deeply connected. We will also deal with hyperbolic polynomials, which share the same field of mathematical root connection as a generalization of real-rootedness. Thus, here we will study the relations between stability, hyperbolicity and RZ-ness in the context of Eulerian polynomials.

We found a stable multivariate generalization of Eulerian polynomials. We will see that this finding gets us indeed closer to the point at which the relaxation becomes useful for the task of finding bounds for the extreme roots of some related polynomials. Now, in order to achieve this, we need to exploit the deep connection between stable polynomials and RZ polynomials. We need to do this because the relaxation introduced in \cite{main} only works in the RZ setting. Happily, moving from stable to RZ polynomials is easy. To do this, we first need to actually introduce the last kind of polynomials that will play a role in our study.

\begin{definicion}[Hyperbolicity]\cite[Definition 6.1]{main}
Let $p\in\mathbb{R}[\mathbf{x}]$ be a polynomial. We say that $p$ is \textit{hyperbolic with respect to a direction} $e\in\mathbb{R}^{n}$ if $p$ is homogeneous, $p(e)\neq0$ and, for any vector $a\in\mathbb{R}^{n}$, we have that $p(a-te)\in\mathbb{R}[t]$ is real-rooted.
\end{definicion}

Hyperbolic and real stable polynomials are related as follows.

\begin{proposicion}[Real stability iff hyperbolicity in all positive directions]\cite[Proposition 1.3]{pemantle2012hyperbolicity} and \cite[Section 5]{kummer2015hyperbolic}
A homogeneous polynomial $p\in\mathbb{R}[\mathbf{x}]$ is real stable if and only if, for any direction $e\in\mathbb{R}^{n}_{>0}$ with positive coordinates and $a\in\mathbb{R}^{n}$, the univariate polynomial $p(a-te)$ is real-rooted. 
\end{proposicion}

This immediately provides the following characterization.

\begin{corolario}[Homogeneous setting]
\label{rstoposorth}
A real homogeneous polynomial is real stable if and only if it is hyperbolic in every direction in the positive orthant.
\end{corolario}

Now we just need to mention the following equivalence that we can find, e.g., in \cite[Proposition 6.7]{main}.

\begin{proposicion}[Hyperbolicity and RZ-ness]
\label{dehomo}
Let $p\in\mathbb{R}[x_{0},\mathbf{x}]$ be a homogeneous polynomial. Then $p$ is hyperbolic in the direction of the first unit vector $u=(1,\mathbf{0})\in\mathbb{R}^{n+1}$ if and only if its dehomogenization $q=p(1,x_{1},\dots,x_{n})\in\mathbb{R}[\mathbf{x}]$ is a RZ polynomial.
\end{proposicion}

As a consequence, we can prove now the following result.

\begin{corolario}[Real stability and RZ-ness]
Let $p\in\mathbb{R}[x_{0},\mathbf{x}]$ be a homogeneous real stable polynomial such that $p(u)\neq0$ for $u=(1,\mathbf{0})\in\mathbb{R}^{n+1}$ the first unit vector. Then its dehomogenization $q=p(1,x_{1},\dots,x_{n})\in\mathbb{R}[\mathbf{x}]$ is a RZ polynomial.
\end{corolario}

\begin{proof}
Using Proposition \ref{rstoposorth} we get that $p$ is hyperbolic with respect to any vector $e\in\mathbb{R}_{>0}^{n+1}$. This means that $p((x_{0},x)-te)$ is real-rooted for any $e\in\mathbb{R}_{>0}^{n+1}$ and $(x_{0},x)\in\mathbb{R}^{n+1}$. Taking the limit when $e\to u$ this implies that $p((x_{0},x)+tu)$ is real-rooted for any $(x_{0},x)\in\mathbb{R}^{n+1}$. Which means that $p$ is hyperbolic in the direction of $u$ because $p(u)\neq0$ by hypothesis. Thus, by the Proposition \ref{dehomo} above, its dehomogenization $q:=p(1,x_{1},\dots,x_{n})\in\mathbb{R}[x]$ is an RZ polynomial.
\end{proof}

This let us ready to apply our knowledge of the relations between these properties to the multivariate version of the Eulerian polynomial introduced above in Definition \ref{multieulerian}. For this, we focus our attention in the exercise of translating these real stable multivariate Eulerian polynomials into the RZ setting. In order to do this, we will study now the degree of the monomials in the expansion defining the polynomials introduced in Definition \ref{multieulerian}. We first ask whether these polynomials are already homogeneous and \textit{in what sense}. We have the following result.

\begin{proposicion}[Real stability of Eulerian polynomials]
$A_{n}(\mathbf{x},\mathbf{y})$ is homogenenous of degree $n$ and real stable.
\end{proposicion}

\begin{proof}
We can see this analyzing the tagged variables through the corresponding permutations. For each term $(\sigma_{1},\dots,\sigma_{n+1})=\sigma\in\mathfrak{S}_{n+1}$ observe that, for each pair $(\sigma_{i},\sigma_{i+1})$, we have that either $\sigma_{i}>\sigma_{i+1}$ or $\sigma_{i}<\sigma_{i+1}$. Hence, for each such pair indexed by $i\in[n]$, we add a variable to the corresponding monomial. As there are $n$ of these pairs, all monomials have the same degree $n$ in the variables $(\mathbf{x},\mathbf{y})$.
\end{proof}

Now, in order to get closer to a RZ polynomial, we will discard counting ascents. We do this by discarding the counting of ascents because, by forgetting about ascents in our computations, we will obtain polynomials whose value is different from $0$ at the origin. This is precisely a condition that we require for RZ polynomials.

\begin{convencion}[Forgetting ascents]
We set all the $\mathbf{y}$ variables equal to $1$, i.e., $\mathbf{y}=\mathbf{1}$.
\end{convencion}

Now we study what happens when we dehomogenize by the first unit vector in order to obtain a RZ polynomial. Thus we continue with a dehomogenization via forgetting the ascents. We find a middle step in this process of translation of settings.

\begin{observacion}[Setting the ascent top variables]
We will consider the polynomial obtained when $\mathbf{y}=\mathbf{1}$. This polynomial is clearly also real stable by \cite[Lemma 2.4]{wagner2011multivariate}. Moreover, it is easy to see that $A_{n}(\mathbf{x},\mathbf{y})$ equals the homogenization of $A_{n}(\mathbf{x},\mathbf{1})$ if we fix all the variables in the tuple $\mathbf{y}$ equal to the same variable $y$.
\end{observacion}

In fact, by counting permutations it is easy to see that $A_{n}(\mathbf{0},\mathbf{1})=1$, which is something we want. Now we have the polynomial we want. We only have to prove the following.

\begin{corolario}[RZ-ness of dehomogenization]
$A_{n}(\mathbf{x},\mathbf{1})$ is RZ.
\end{corolario}

\begin{proof}
$A_{n}(\mathbf{x},\mathbf{1})$ is the dehomogenization with respect to the unit vector corresponding to the homogenizing variable $y$ of $A_{n}(\mathbf{x},y,\dots,y)$. Using Corollary \ref{rstoposorth} and that $A_{n}(\mathbf{x},\mathbf{y})$ is real stable we get therefore that $A_{n}(\mathbf{x},y,\dots,y)$ is hyperbolic with respect to the unit vector corresponding to the homogenizing variable $y$. Therefore, finally, using Proposition \ref{dehomo} we get that that $A_{n}(\mathbf{x},\mathbf{1})$ is RZ, which finishes this proof.
\end{proof}

Now we finally have a multivariate lifitng of the Eulerian polynomials that has a form acceptable for the relaxation. We are just a few computations away from being able to construct the spectrahedral relaxations of their rigidly convex sets introduced in \cite{main}. We collect the results of these computations in the next section.

\section{Counting permutations in terms of descents to compute $L$-forms}\label{count}

The objects over which we have to count here are sets of permutations verifying certain properties written in terms of their descents. In this regard, the most important sets are these defined by fixing an exact descent top set.

\begin{definicion}[Exact descent]
\label{rns}
Fix $S\subseteq[n+1]$. We denote the set of permutations that \textit{descend exactly} at $S$ by $$R(n,S):=\{\sigma\in\mathfrak{S}_{n+1}\mid S=\mathcal{DT}(\sigma)\}.$$
\end{definicion}

We will see that the cardinals of these sets are the basic objects that will appear when we want to compute the entries of the coefficient matrices forming the LMP defining the relaxation. In particular, because the relaxation only uses the up to degree $3$ part of the polynomials, we will always have the inequality $|S|\leq3.$ Also, in order to shorten the notation, when $n$ is fixed we simply denote these numbers $R(S)$ for any subset $S\subseteq[n+1].$ Before computing the relaxation, it is better to determine first the cardinalities of these sets. For this, we will use general expressions for these cardinalities in easily computable terms. Before introducing these formulas we need to refine the concepts introduced in Definition \ref{defexcat}.

\begin{definicion}[Pairs]\cite[Section 1]{hall2008counting}
Let $\sigma=(\sigma_{1}\cdots\sigma_{n+1})\in\mathfrak{S}_{n+1}$ be a permutation. A \textit{descent pair} of $\sigma$ is a pair $(\sigma_{i},\sigma_{i+1})$ for $i\in[n]$ such that $\sigma_{i}>\sigma_{i+1}.$ We define an \textit{ascent pair} analogously.
\end{definicion}

For counting purposes, we will study pairs whose top lies in some fixed set $X$ and whose bottom lies in another fixed set $Y$. In order to keep the discussion short and because of the definition of our RZ multivariate liftings, we will describe this just for descent pairs.

\begin{definicion}[Adequate descents]\cite[Definition 1.1]{hall2008counting}
Fix subsets $X,Y\subseteq\mathbb{N}$ and a permutation $\sigma\in\mathfrak{S}_{n+1}.$ We define the set of \textit{adequate descents} of $\sigma$ for the pair of sets $(X,Y)$ as $$\Des_{X,Y}(\sigma):=\{i\in[n]\mid\sigma_{i}>\sigma_{i+1}\mbox{\ and\ }\sigma_{i}\in X \mbox{\ and\ } \sigma_{i+1}\in Y\}.$$
\end{definicion}

Fortunately, there are known formulas for computing the number of these permutations. We introduce the precise objects that these formulas can compute in the notation below.

\begin{notacion}[Number of adequate descents]\cite[Theorems 2.3 and 2.5]{hall2008counting}
Denote $P^{X,Y}_{n,s}$ the number of permutations $\sigma\in\mathfrak{S}_{n}$ with at least $s$ adequate descents for the pair $(X,Y)$.
\end{notacion}

This notation is trickier than it looks. We have to be sure that the reader understands the permutations computed through these numbers.

\begin{warning}[Exactness in descents]\label{warningexact}
Beware that $P_{n,s}^{X,Y}$ counts \textit{all} the permutations $\sigma\in\mathfrak{S}_{n}$ with \textit{at least} $s$ adequate descents for the pair $(X,Y)$. This means that it counts permutations having also other additional descents. In the future, we will need to refine these numbers because we want to be able to count the number of permutations $\sigma\in\mathfrak{S}_{n}$ with \textit{exactly} $s$ adequate descents \textbf{and no others} in order to find the coefficients of our polynomials.
\end{warning}

Here we will be mainly looking at the case where we let unrestricted the set $Y$ so that $Y=\mathbb{N}$. Thus, the bottoms of the descent pairs stays unrestricted.

\begin{particularizacion}[Unrestricted bottom set]
In this case, we denote $P_{n,s}^{X}:=P_{n,s}^{X,\mathbb{N}}$ the number of $\sigma\in\mathfrak{S}_{n}$ with at least $s$ adequate descents for the pair $(X,\mathbb{N})$.
\end{particularizacion}

We need a couple of technical functions that will help us in our task of counting. These functions are used in the next theorem allowing us to determine formulas for the sets $P_{n,s}^{X}$ introduced above.

\begin{definicion}[Technical counters]\cite[Theorems 2.3 and 2.5]{hall2008counting}
For $A\subseteq\mathbb{N}$ and $n\in\mathbb{N}$ we write $A_{n}:=A\cap[n]$ and $A_{n}^{c}:=(A^{c})_{n}=[n]\smallsetminus A$. Additionally, for $j\in[n],$ we denote \begin{gather*}
\alpha_{A,n,j}:=|A^{c}\cap\{j+1,\dots,n\}|=|\{x\in A^{c}\mid j<x\leq n\}| \mbox{\ and\ }\\
\beta_{A,n,j}:=|A^{c}\cap\{1,\dots,j-1\}|=|\{x\in A^{c}\mid 1\leq x<j\}|,
\end{gather*} where the complement operator is taken with respect to $\mathbb{N}$ so $A^{c}=\mathbb{N}\smallsetminus A.$
\end{definicion}

Now we can count the number of $n$-permutations with a fixed number $s$ of adequate descent pairs for the pair $(X,\mathbb{N})$. As $Y$ is here unrestricted, this computations will give us the number of permutations with a fixed number of descents $s$ whose top lies in $X$.

\begin{teorema}[Formula for the number of $n$-permutations with a fixed number of descents and whose tops are in a fixed set]\cite[Theorems 2.3 and 2.5]{hall2008counting} $P_{n,s}^{X}=$\begin{gather*}
|X^{c}_{n}|!\sum_{r=0}^{s}(-1)^{s-r}\binom{|X^{c}_{n}|+r}{r}\binom{n+1}{s-r}\prod_{x\in X_{n}}(1+r+\alpha_{X,n,x})=\\
|X^{c}_{n}|!\sum_{r=0}^{|X_{n}|-s}(-1)^{|X_{n}|-s-r}\binom{|X^{c}_{n}|+r}{r}\binom{n+1}{|X_{n}|-s-r}\prod_{x\in X_{n}}(r+\beta_{X,n,x}).
\end{gather*}\end{teorema}

As we noted above, this theorem is a tool towards our goal but it is not yet enough for us. This is so because of the appearance of additional descents due to the fact that, as we said before, the map $P_{n,s}^{X}$ counts \textit{all} the permutations with at least $s$ descents with tops in the set $X$. However, these permutations might have other \textit{additional} descents. It is evident that we need a formula that computes the permutations with no other descents than the $s$ whose descent tops are in $X$. Thus, setting $s=|X|$, the number $P_{n+1,s}^{X}=P_{n+1,|X|}^{X}$ will actually not be the cardinal of the set $$R(n,X):=\{\sigma\in\mathfrak{S}_{n+1}\mid X=\mathcal{DT}(\sigma)\}$$ because $P_{n+1,|X|}^{X}$ counts some extra permutations. Therefore we need to count with more care noticing that $P_{n+1,|X|}^{X}$ is the cardinal, already introduced in Definition \ref{defexcat}, of the set $$(\mathfrak{S}_{n+1})(X):=\{\sigma\in\mathfrak{S}_{n+1}\mid X\subseteq\mathcal{DT}(\sigma)\}.$$ Nevertheless, considering now $P_{n+1,|X|}^{X}$, we can complete the remaining details producing a more suitable formula in the following result. Thus, because for our purposes here we want to fill the whole restriction set $X$ with \textit{actual} descents (i.e., we want $s=|X|$ and $X\subseteq[n]$), we find the following direct corollary to be a useful tool for performing our computations.

\begin{corolario}[Totally realized fixed descent top set]\label{xordenadas}
In our setting, where $s=|X|$ and $\{x_{1}<\cdots<x_{s}\}=X\subseteq[n]$, we have that $P_{n,|X|}^{X}=$
\begin{gather*}
(n-|X|)!\prod_{i=1}^{s}(x_{i}-i).\end{gather*}
\end{corolario}

\begin{proof}
\begin{gather*}
(n-|X|)!\sum_{r=0}^{|X|-|X|}(-1)^{|X|-|X|-r}\binom{n-|X|+r}{r}\binom{n+1}{|X|-|X|-r}\prod_{x\in X}(r+\beta_{X,n,x})=\\(n-|X|)!\sum_{r=0}^{0}(-1)^{-r}\binom{n-|X|+r}{r}\binom{n+1}{-r}\prod_{x\in X}(r+\beta_{X,n,x})=\\(n-|X|)!\prod_{x\in X}(\beta_{X,n,x})=\\(n-|X|)!\prod_{x\in X}(|\{y\in [n]\smallsetminus X\mid 1\leq y<x\}|)=(n-|X|)!\prod_{i=1}^{s}(x_{i}-i),\end{gather*}
as we indexed the set $X$ according to its order so $x_{1}<\cdots<x_{s}.$ \end{proof}

Now that the simplifications above allowed us to construct this shorter formula adapted to our setting, we can address the problem mentioned above in Warning \ref{warningexact} in a constructive way. In this way, we will explore formulas suited for computing the number of permutations having exactly $s$ descents whose tops are in $X$ with $|X|=s$ and no other descents beyond the ones in $X$. In order to construct these useful formulas, we will combine the last Corollary \ref{xordenadas} and the well-known inclusion-exclusion principle. For this, we first have to examine the sets to which we need to apply this principle. Mirroring Definition \ref{defexcat}, we obtain the following important sets.

\begin{notacion}[Mirror]\label{defe}
We denote $$(X)(\mathfrak{S}_{n+1}):=\{\sigma\in\mathfrak{S}_{n+1}\mid\mathcal{DT}(\sigma)\subseteq X\}.$$
\end{notacion}

We also introduce a couple of maps on ordered sets and tuples that will be useful for computing the cardinals we want.

\begin{notacion}[Operators]\cite[Appendix A]{davis2018pinnacle}
Fix an ordered set $X=\{x_{1}<\dots<x_{k}\}$. We denote $$\alpha(X)=(x_{1}-1,x_{2}-x_{1},x_{3}-x_{2},\dots,x_{k}-x_{k-1}).$$ Fix a tuple $\beta=(\beta_{1},\dots,\beta_{k})$. We define the operator $$\beta\hat{!}=(k+1)^{\beta_{1}}k^{\beta_{2}}\cdots3^{\beta_{k-1}}2^{\beta_{k}}.$$
\end{notacion}

We convene $\alpha(\emptyset)=()$ and $()\hat{!}=1$ for completeness. Now we can write down the cardinal of the sets of permutations considered in Notation \ref{defe}.

\begin{proposicion}\cite[Theorem A.1]{davis2018pinnacle}
Let $X\subseteq[n]$. Then $$|(X)(\mathfrak{S}_{n})|=\alpha(X)\hat{!}$$\end{proposicion}

Using this, we can obtain directly the following two formulas for computing the numbers $|R(n,X)|.$

\begin{corolario}[Cardinal of sets with exact descent top set]
\label{coroR}
Fix $s=|X|$ and $\{x_{1}<\cdots<x_{s}\}=X\subseteq[n]$ and $\{y_{1}<\cdots<y_{n-s}\}=Y\subseteq[n]$ with $X\cup Y=[n]$. For subsets $S\subseteq Y$, name the ordered chain of elements obtained through the union $X\cup S=\{x_{S,1}<\cdots<x_{S,s+|S|}\}$. Thus, going through the complement, we have that $|R(n-1,X)|=$
\begin{gather}\label{coroR1}
\sum_{S\subseteq[n]\smallsetminus{X}}(-1)^{|S|}(n-|X\cup S|)!\prod_{i=1}^{s+|S|}(x_{S,i}-i).
\end{gather} Similarly, we can express this number in terms of deletions in the initial set as $|R(n-1,X)|=$
\begin{gather}\label{coroR2}
\sum_{J\subseteq X}(-1)^{|X\smallsetminus J|}\alpha(J)\hat{!}.
\end{gather}
\end{corolario}

\begin{proof}
We only prove the first identity, the second follows similarly and its detailed proof can be consulted in \cite[Appendix A]{davis2018pinnacle}. Using the inclusion-exclusion principle we have that $$
\sum_{S\subseteq[n]\smallsetminus{X}}(-1)^{|S|}P_{n,|X\cup S|}^{X\cup S},$$ which, after the discussion on the section before, clearly equals $$\sum_{S\subseteq[n]\smallsetminus{X}}(-1)^{|S|}(n-|X\cup S|)!\prod_{i=1}^{s+|S|}(x_{S,i}-i).$$
\end{proof}

Now we are correctly equipped for computing the cardinals that we need to know in order to construct the relaxation of our RZ multivariate liftings of the Eulerian polynomials. As the relaxation only uses the up to degree $3$ part of the polynomials we feed it with, we focus our attention now in these sets $X$ with $|X|\leq3$.

\begin{warning}[Abuse of notation]
\label{abuseofnotationforR}
In this section we are not interested in the sets $R(X)$ themselves. We only need their cardinalities. For this reason, we will abuse the notation in what follows and refer to the cardinal $|R(X)|$ of $R(X)$ simply as $R(X)$.
\end{warning}

Applying the Corollary \ref{coroR} above to the cases $|X|\in\{1,2,3\}$ and fixing $n$ so we can shorten $R(X)=R(n,X)$, we obtain the following values.

\begin{computacion}\label{rx}
$R(X)$ equals
\begin{enumerate} 
    \item $2^{x_{1}-1}-1$ for $X=\{x_{1}\}$.
    \item $3^{x_{1}-1}2^{x_{2}-x_{1}}-(2^{x_{1}-1}+2^{x_{2}-1})+1$ for $X=\{x_{1}<x_{2}\}$.
    \item $4^{x_{1}-1}3^{x_{2}-x_{1}}2^{x_{3}-x_{2}}-(3^{x_{1}-1}2^{x_{2}-x_{1}}+3^{x_{2}-1}2^{x_{3}-x_{2}}+3^{x_{1}-1}2^{x_{3}-x_{1}})+(2^{x_{1}-1}+2^{x_{2}-1}+2^{x_{3}-1})-1$ for $X=\{x_{1}<x_{2}<x_{3}\}$.
\end{enumerate}
\end{computacion}

This was a middle step that let us closer to our objective. Now we just have to calculate the corresponding $L$-forms using these values. As we are still computing, we remind that the Warning \ref{abuseofnotationforR} about the notation $R(X)$ still applies here.

\begin{convencion}[Fixing as Eulerian polynomials]
In order to simplify the notation, we fix, from now on, $n$ and call $p:=A_{n}(\mathbf{x},\mathbf{1})$. Notice that this $p$ is already normalized because we saw that $A_{n}(\mathbf{0},\mathbf{1})=1.$
\end{convencion}

Fix $i,j\in\{0\}\cup[n].$ The value at position $ij$ of the LMP produced by the relaxation is $$\left(\sum_{k\in\{0\}\cup[n]}x_{k}L_{p,d}((x_{k}x_{i}x_{j})|_{x_{0}=1})\right)\bigg{|}_{x_{0}=1}.$$ We will therefore calculate the values of the $L$-form over monomials $m$ of degree up to three. We compute these values using the formulas obtained in \cite[Example 3.5]{main} together with Computation \ref{rx}.

\begin{computacion} $L_{p}(m)$ equals
\begin{enumerate}
    \item $n$ when $m=1$.
    \item $2^{i-1}-1$ when $m=x_{i}$.
    \item $(2^{i-1}-1)^{2}$ when $m=x_{i}^{2}$.
    \item $2^{-2 + i + j} - 2^{-i + j} \cdot 3^{-1 + i}$ when $m=x_{i}x_{j}$ with $i<j$.
    \item $(2^{i-1}-1)^{3}$ when $m=x_{i}^{3}$.
    \item $\frac{1}{3}\cdot 2^{-3 - i + j} (-2 + 2^i) (-4 \cdot 3^{i} + 3 \cdot 4^{i})$ when $m=x_{i}^{2}x_{j}$ with $i<j$.
    \item $\frac{1}{3}\cdot2^{-3 + i - j} (-2 + 2^i) (-4 \cdot 3^j + 3 \cdot 4^j)$ when $m=x_{i}^{2}x_{j}$ with $j<i$.
    \item $2^{-3 + i + j + k} - 2^{-1 - i + j + k} \cdot 3^{-1 + i} - 2^{-2 + i - j + k} \cdot 3^{-1 + j} + 2^{-3 + 2 i - j + k} \cdot 3^{-i + j}$ when $m=x_{i}x_{j}x_{k}$ with $i<j<k$. 
    \end{enumerate}
\end{computacion}

All these cases cover all the possible values of $L_{p}$ over the monomials indexing the mold matrices forming the relaxation. Notice that we had to divide one of the degree $3$ cases above in two different cases because the symmetry was broken. Now we can use directly these computations of the relevant values of the $L$-forms without a necessity to think or perform computations in terms of sets of permutations. Now everything is arithmetic and algebra, no more combinatorics. This makes our work much easier because now we can directly use these values in our computer software in order to continue working.

Now we can build the relaxation and continue. However, in order to understand how much better is the bounding information provided by the relaxation when the number of variables increase in contrast to what happens when we stay in the univariate setting, we will analyze the asymptotic behaviour and the overall form of the bounds obtained through the relaxation. Following this exercise, we will compare the bounds obtained for the extreme roots of Eulerian polynomials we are interested in through the univariate setting and through the multivariate setting. We will see that, in fact, going multivariate gives an edge for our estimations and provides better bounds for these roots. This means that lifting the Eulerian polynomials to the multivariate setting allows us to collect more information about these through the relaxation. Or, at least, the relaxation is able to recover more information about the extreme roots of these polynomials when we deal with more variables. This is the topic of the next sections.

\section{Recovering previous bounds}\label{rec}

When we apply the relaxation to univariate Eulerian polynomials we obtain a bound that has appeared previously in the literature. In fact, we obtain the Colucci estimation for these polynomials.

\begin{teorema}[Colucci estimation]\cite[Subsection 7.6.5]{mezo2019combinatorics}
Let $p(x)=\sum_{i=0}^{n}a_{i}x^{i}\in\mathbb{C}[x]$ be a polynomial whose zeros are bounded in absolute value by some positive real number $M\in\mathbb{R}_{\>0}$ and fix $k\in\{0,\dots,n\}$. Then we can bound the absolute value of the $k$-th derivative of $p$ as \begin{gather*}\label{ineqcol}|p^{(k)}(x)|\leq k!\binom{n}{k}|a_{n}|(|x|+M)^{n-k}.\end{gather*}
\end{teorema}

We will need a notation for the sequences of numbers that define our polynomials. These numbers form a triangle in a similar way as binomial coefficients form Pascal's triangle. This Eulerian number triangle has also very interesting properties. See, e.g., \cite{GneOls06,zhu2020generalized,maier2023triangular} for some of these interesting properties of the Eulerian triangle.

\begin{notacion}[Eulerian numbers]\cite{petersen2015eulerian}
We expand the $n$-th univariate Eulerian polynomial as $A_{n}(x)=\sum_{k=0}^{n}E(n+1,k)x^{k}$. The coefficients $E(n,k)$ are called \textit{Eulerian numbers}.   
\end{notacion}

The asymptotic growth of the Eulerian numbers is interesting for our purposes here. In order to study this growth, we need adequate expressions for these numbers. The next way of writing them is well-known and useful for our purposes.

\begin{proposicion}[Expansion]\cite{petersen2015eulerian,mezo2019combinatorics}\label{expansionEuler}
The Eulerian numbers admit the expansion $$E(n+1,k)=\sum_{i=0}^{k}(-1)^{i}\binom{n+2}{i}(k+1-i)^{n+1}.$$ Moreover, $E(n+1,k)=E(n+1,n-k)$.
\end{proposicion}

\begin{proof}
This is just a count of the number of permutations in $\mathfrak{S}_{n+1}$ having $k$ and $n-k$ descents looking at obvious symmetries.
\end{proof}

Now we can easily establish the asymptotic growth of these numbers just looking at how we wrote them above.

\begin{proposicion}[Asymptotic growth of Eulerian numbers]
Let $i\in\mathbb{N}$. Then $\lim_{n\to\infty}\frac{E(n+1,n-k)}{(k+1)^{n+2}}=1$.
\end{proposicion}

\begin{proof}
It is clear looking at the expansions of these numbers as sums of a fixed number of summands. We do this combining both parts of Proposition \ref{expansionEuler} in order to fix the number of summands through symmetry and expanding the rest. After doing that, the asymptotic behaviour becomes clear using that $E(n+1,n-k)=E(n+1,k).$
\end{proof}

Now a simple analysis gives a way to do a transformation providing the bound above. In particular, we obtain that bound using the theorem by rearranging the terms in the Inequality \ref{ineqcol}. Thus we see that this result allows us to estimate a majorant $M$ just using information about the derivatives of $p$ at some points. This is how this theorem is used to obtain the following majorant exposed in \cite[Subsection 7.6.5]{mezo2019combinatorics}. The minorant is obtained through the use of the well-known Laguerre-Samuelson theorem \cite{samuelson1968deviant,jensen1999laguerre,niezgoda2007laguerre}.

\begin{proposicion}[Known bounds]\cite[Subsection 7.6.5]{mezo2019combinatorics}
\label{clearerasin}
We have the chain of inequalities \begin{gather*}I(n):=\frac{1}{n}E(n+1,n-1)+\frac{n-1}{n}\sqrt{E(n+1,n-1)^2-\frac{2n}{n-1}E(n+1,n-2)}=\\\frac{2^{n+1}-n-2}{n}+\frac{n-1}{n}\sqrt{(2^{n+1}-n-2)^2-\frac{2n}{n-1}(3^{n+1}-(n+2)2^{n+1}+\frac{1}{2}(n+1)(n+2))}\\\geq |q_{1}^{(n)}|\geq\frac{2^{n+1}}{n}-1-\frac{2}{n}\end{gather*} with $\lim_{n\to\infty}\frac{I(n)}{2^{n+1}}=1.$
\end{proposicion}

\begin{proof}
We obtain the asymptotic equivalence from the identity $$\lim_{n\to\infty}\frac{E(n+1,n-1)} {2^{(n+1)}}=1=\lim_{n\to\infty}\frac{E(n+1,n-2)}{3^{(n+1)}}$$ so it is clear that, when $n\to\infty$, \begin{gather*}
   \frac{I(n)}{2^{n+1}}\to 0+1\sqrt{1-2(0-0+0)}=1.
\end{gather*}
\end{proof}

Now we will see that we can use our relaxation in order to find Colucci's bounds instead. This will show that the relaxation hides inside it the ability to provide different families of bounds. In the future we will see that it can even give better bounds when we make the correct choices and conserve enough information along the way. This is so because we do not even require the whole power of the relaxation to recover Colucci's estimation. This estimation is already just literally \textit{in a corner}.

We are going to stay all the time now in the RZ setting because this is the most suitable setting for using the relaxation. Remember that the relaxation needs to be fed with RZ polynomials. For this reason, we can shorten our names.

\begin{notacion}[Simplified name for multivariate Eulerian polynomials]
From now on, we will denote $A_{n}(\mathbf{x}):=A_{n}(\mathbf{x},\mathbf{1})$, forgetting directly about the $\mathbf{y}$ variables encoding ascent information, as we will not use them anymore.
\end{notacion}

We can easily notice that looking only at the top left entry of the relaxation gives already Colucci's estimation. We explore in detail how this is done.

\begin{observacion}[Relaxation allows Colucci's estimation]
Applying the relaxation to any RZ polynomial $p\in\mathbb{R}[\mathbf{x}]$, considering just its top left entry and imposing that this entry has to be positive (or, equivalently using the vector $(1,\mathbf{0})$) gives directly the inequality $L_{p}(1)+x\sum_{i=2}^{n+1}L_{p}(x_{i})=n+x\sum_{i=2}^{n+1}(2^{i-1}-1)=n+x(-2 + 2^{1 + n} - n)\geq0$. Remembering that univariate Eulerian polynomials are palindromic, we can see 
that this inequality translates into Colucci's bound. This happens in this way: $q_{1}^{(n)}\leq\frac{-2 + 2^{1 + n} - n}{-n}$ or, in absolute value, $|q_{1}^{(n)}|\geq\frac{2^{1 + n}}{n}-\frac{2}{n}-1$, where $q_{1}^{(n)}$ is the root of $A_{n}(x)$ having the largest absolute value. This last inequality is \textit{exactly} Colucci's estimation, as we saw in Proposition \ref{clearerasin}.
\end{observacion}

Here we saw an example of how to recover a previously obtained (in \cite{mezo2019combinatorics}) bound. Now we want to explore how to obtain more and better bounds. In order to do that, we have to learn how to work with the relaxation to obtain root bounds. We explore this in the next section.

\section{Extracting bounds from the relaxation}\label{ext}

There are two ways of practically extracting bounds from the relaxation. If the LMP is small, one can try to just solve it. This is possible when we apply the relaxation to univariate polynomials. However, once the size of matrices in the LMP begins to grow, we have to find other strategies. The main strategy we employ here is trying to figure out a good candidate sequence of approximated generalized eigenvectors for the associated generalized eigenvalue problem. See, e.g., \cite[Appendix A]{zienkiewicz2005finite} for more on these problems.

We will see this two ways of proceeding soon. First we analyze how to deal with these \textit{approximated} generalized eigenvalue problems because they will help us overcome the difficulties carried by multivariability when the matrix coefficients of the LMP defining the relaxation become unmanageably increasing. We introduce notation for the LMP of our interest.

\subsection{Linear bounds by approximating generalized eigenvectors}\label{linearbounds}

\begin{notacion}[LMP obtained]\label{LMP}
Remember that, for multivariate Eulerian polynomials, the tuple $\mathbf{x}=(x_{1},\dots,x_{n+1})$ but $x_{1}$ acts as a ghost variable that never actually appears as all monomials containing it have coefficient $0$ due to the fact that $1$ is never a top. The relaxation of the $n$-th multivariate Eulerian polynomial $A_{n}(\mathbf{x})\in\mathbb{R}[\mathbf{x}]$ provides us therefore with a matrix polynomial $$\Tilde{M}_{n}(\mathbf{x}):=\Tilde{M}_{n,0}+\sum_{j=1}^{n+1}x_{j}\Tilde{M}_{n,j}=\Tilde{M}_{n,0}+\sum_{j=2}^{n+1}x_{j}\Tilde{M}_{n,j}\in\sym_{n+2}[\mathbf{x}].$$ As $x_{1}$ is a ghost variable, $\Tilde{M}_{n,1}=0$ and, similarly, the second row and column of $\Tilde{M}_{n,1}$ (these corresponding to indexing by the monomials $x_{1}$ in the moment matrix forming the mold of the relaxation) is constantly $0$ for all $j\in[n+1].$ For this reason, we prefer to consider the matrices obtained from cutting away these rows and columns. Thus we obtain $$M_{n}(\mathbf{x}):=M_{n,0}+\sum_{j=1}^{n+1}x_{j}M_{n,j}=M_{n,0}+\sum_{j=2}^{n+1}x_{j}M_{n,j}\in\sym_{n+1}[\mathbf{x}].$$
\end{notacion}

As $A_{n}$ is RZ and because of the properties of the relaxation proved in \cite{main}, that reduced LMP is positive semidefinite at the origin, i.e., $M_{n,0}$ is PSD. Setting $x_{i}=x$ for all $i\in[n+1]$ we look at the diagonal, where we find the original univariate Eulerian polynomial $A_{n}(x)\in\mathbb{R}[x]$. This will be important in Section \ref{ame}, when we talk about what happens in this diagonal as a measure of accuracy for our bounds. For this reason, we introduce now a notation for this.

\begin{notacion}[LMP in the diagonal]\label{LMPdia}
Looking through the diagonal we obtain the univariate LMP $$M_{n}(x,\dots,x)=M_{n,0}+x\sum_{i=2}^{n+1}M_{n,i}:=M_{n,0}+xM_{n,\suma}\in\sym_{n+1}[\mathbf{x}].$$ 
\end{notacion}

The determinant of this LMP is a univariate polynomial whose largest root bounds from below the largest root $q_{n}^{(n)}$ (which is always negative) of the corresponding univariate Eulerian polynomial $A_{n}(x,\dots,x)$. Now, if instead we want to form a linear bound, we need to study the behaviour of the kernel of the matrix obtained when $x$ is the largest root $x_{n,r}$ of $\det(M_{n,0}+xM_{n,\suma})$. This largest root $x_{n,r}$ is negative as it must indeed verify $x_{n,r}\leq q_{n}^{(n)}<0$ because of the property of being a relaxation.

In order to study such kernel, we will have to experiment in order to construct good numerical guesses giving approximate generalized eigenvectors associated to these largest generalized eigenvalues. Thus, the approximate analysis of this kernel will be done through a construction producing many different linear inequalities. We explain in detail such procedure with the intention to make clear the reason why it works.

\begin{procedimiento}[Obtaining inequalities]
\label{procesobound}
In particular, we will apply the following criterion in order to build inequalities using our procedure. A LMP $$p(x)=A+xB\in\sym_{n}(\mathbb{R})[x]$$ with $A$ a PSD matrix verifies $$v^{\top}(A+xB)v=v^{\top}Av+xv^{\top}Bv\geq0$$ for all $v\in\mathbb{R}^{n}$ and $x\in\rcs(\det(p))$. Hence, for any $v\in\mathbb{R}^{n}$, this choice provides us with inequalities of the form $a+xb\geq0$ (one for each $v$). These inequalities can be rearranged solving for $x$ into $x\geq\frac{-a}{b}$ whenever $b>0$. This lower-bounds the set $\rcs(\det(p))$. In other words, $x\in\rcs(\det(p))$ must verify $x\geq\frac{-a}{b}.$
\end{procedimiento}

We have to be careful about some subtleties that pop up when we deal with inequalities. In particular, we must ensure that these do not change direction.

\begin{remark}[Avoid reverse inequalities]
Extreme care is required in the last part of the procedure: we need $b>0$ in order to lower-bound through this method. In other case, if $b<0$, the inequality gets reversed.
\end{remark}

This is precisely how we recovered exactly the previous Colucci estimation choosing a vector $(1,\mathbf{0})$ selecting exactly the top left entry of the matrix. Now we can generalize this procedure to obtain other families of \textit{linear} bounds. But before we will see the other way to obtain bounds from the relaxation when the matrices are small enough.

\subsection{Exact bounds solving determinant in small number of variables}\label{exactbounds}

The other way of obtaining bounds consists simply in computing the determinant of the LMP defining the relaxation. If the RZ polynomial does not have many variables, the matrices will be small enough to form a determinant of low degree. When we restrict to the diagonal again this polynomial becomes univariate. We have many (see, e.g.,\cite{davenport1990finding,anari2018approximating} and, for a highly comprehensive compilation, \cite{macnameeii}) tools to approximate or even calculate the roots of these polynomials. The largest root of that polynomial will again be an outer bound for the largest root of the original polynomial.

This is a natural way to extract bounds from the relaxation, but it is not much more advantageous than just looking at the original sequence of polynomials if the degree grows at the same pace as the number of variables does. For univariate Eulerian polynomials, as the number of variables is stably $1$ while the degree grows, we have a nice advantage when using the relaxation. This is the topic of the next section.

\section{Asymptotics of spectrahedral bounds before lifting}\label{asy}

Here, \textit{before lifting} means looking at the univariate Eulerian polynomials, that is, \textit{before lifting} to the multivariate RZ setting. In this case we can proceed computing the determinant and solving it, as we mentioned at the end of previous section. Doing this will be surprisingly advantageous for extracting information about the roots because, as we will see here, this procedure will solve a question that was asked in \cite{mezo2019combinatorics}.

We therefore claim that an even more straightforward application of our method already asymptotically improves the previously obtained lower bound. Moreover, in this way we also get a tight first order asymptotic estimation for such root. Thus, we answer, through this method, the next question of Mez\H{o}.

\begin{cuestion}[Mez\H{o}]\cite[Subsection 7.6.5]{mezo2019combinatorics}\label{mezo}
Provide $\beta\in[0,1]$ and $d>0$ such that $$\lim_{n\to\infty}\frac{n^\beta|q^{(n)}_{1}|}{2^{n+1}}=d.$$
\end{cuestion}

We originally began researching the application of the relaxation through this question. But Sobolev already answered this question years ago. His answer can be consulted in \cite{sobolev2006selected}. Although this question was answered there, dealing with this question as if it was open actually allowed us to understand better the relaxation. Moreover, what we obtained here are actual bounds, which is better than just an estimation.

\begin{observacion}[Bound better than estimation]
Thinking about that asymptotic estimation from the perspective of applying the relaxation allowed us to produce insightful results that go beyond an answer to the question posed there. For example, the relaxation provides more than an estimation of the asymptotics, but actual bounds closing the first asymptotic gap with the previous bounds exposed above obtained from using Colucci's estimation.
\end{observacion}

Additionally, we could explore the behaviour of the relaxation when the variables of our polynomial liftings increased. This shows us the benefit of injecting the univariate polynomials we are interested in into multivariate liftings obtained through any nice mechanism allowing us to keep some generalized multivariate notion of real-rootedness, like these we have seen above (real-stability, real-zeroness, hyperbolicty).

\begin{observacion}[Changes of the relaxation under variable increase]
Thinking about this question as an open question let us also explore the properties of the relaxation and the changes it experiences when we increase the number of variables with respect to the task of bounding the roots of univariate Eulerian polynomials injected in the diagonal of multivariate Eulerian polynomials. We will make all these unexpected benefits from our explorations clear now.
\end{observacion}

Hence we see the benefits of \textit{just taking this path}, which is a nice mathematical exploration by itself. We recall some wise words in this regard.

\begin{comentario}[Question already answered]
For the reason exposed above, the way towards our answer is actually more important than the answer itself. This is so because this answer was, in any case, already hidden in the literature about univariate Eulerian polynomials, but, as Gauss famously said in the letter to Wolfgang Bolyai \cite{gauss1808bolyai}, ``it is not knowledge, but the act of learning, not the possession of but the act of getting there, which grants the greatest enjoyment.'' In this very sense, our findings here are not just the knowledge of the asymptotic growth of the extreme roots of the univariate Eulerian polynomials, not an answer to the original question of Mez\H{o}, but the procedures, tools and methods by which we reached that answer. This is what gives us the satisfaction of contemplating the development of our techniques to provide us, not just to an answer by another path, but an insight into the properties, the behaviour, the weaknesses and the strengths of the tools we are testing here. These techniques provided our particular path to obtain another way of answering such question. In particular, the application of the relaxation during our struggle to answer this question showed us many properties about the advantages and limitations of this tool and, \textit{additionally}, also eventually allowed us to solve the question in a different way, following another path.
\end{comentario}

Moreover, we should mention two advantages of our approach. These are related to multivariability and laterality of the bounds.

\begin{observacion}[Further advantages of our path to the answer]
Contrary to obtaining just point estimations, as these given in the previous literature, here we care about actual bounds. That is, we record how these estimations compare to the actual bound. This is what we call laterality. This laterality of our bounds allow us to provide (with the help of bounds by the other side) an actual interval where the quantity we care about (the leftmost root of the univariate Eulerian polynomial) actually lies. Additionally, and centrally for the spirit of this paper, the bound we obtain here are not \textit{unidimensional}. We obtain actual \textit{ovaloid bounds} for the ovaloids defined by the extreme roots of the multivariate Eulerian polynomials (when we move) along each possible line passing through the origin. We are therefore truly accessing the multivariate nature of the question thanks to the multivariate nature of the relaxation, which gives as this power as tool to explore the extreme roots of our polynomials in a multidimensional sense (along every line passing through the origin).
\end{observacion}

We just wanted to point out the strengths of this path and the further insights that following it provides to our overview of the sequences of Eulerian polynomials. However, these observations lead us immediately to phenomena that lies out of our scope in this paper but that it is worth mentioning now to point at future venues of exploration that have just opened here.

\begin{remark}[On attaching weights to descent tops and its relation to lines away from the diagonal]
Over each one of these lines passing through the origin we obtain actually variants of the univariate Eulerian polynomials in which the descent tops have been equipped with a weight. Hence it turns out that we can therefore use similar techniques to study (univariate) generalizations of univariate Eulerian polynomials that emerge naturally when we look, in an analogous way, at these weighted permutations. Weighted permutations clearly generalize colored permutations.
\end{remark}

Our answer to Question \ref{mezo} is that $\beta=0$ and $d=1$. We will see this here. For this, we need to compute the $L$-forms of the univariate Eulerian polynomials.

\begin{computacion}[$L$-forms of the univariate polynomial up to degree $3$]
Now the values are much easier to compute. \begin{enumerate}
    \item $L_{p}(1)=n$.
    \item $L_{p}(x)=2^{n+1}-(n+2)$.
    \item $L_{p}(x^2)=2 - 3^{1 + n}2 + 4^{1 + n} + n$.
    \item $L_{p}(x^3)=-2 - 2^{1 + n} 3^{2 + n} + 4^{1 + n}3 + 8^{1 + n} - n.$
\end{enumerate}
\end{computacion}

Now we can continue. We just have to build the relaxation of the univariate Eulerian polynomial. Doing this we are going to obtain the bound given by the next sequence.

\begin{definicion}[Univariate bound $\un$]
We define the sequence $\un\colon\mathbb{N}\to\mathbb{R}, n\mapsto\un(n):=\frac{2a}{b-\sqrt{b^2-4ac}}$ with $c+bx+ax^{2}$ the determinant $\det(M_{p}(x))$ of the LMP $M_{p}(x)$ obtained by applying the relaxation to $p$ the $n$-th univariate Eulerian polynomial.
\end{definicion}

We will immediately compute the sequence $\un$ explicitly. Thanks to that computation, we will establish how good (in terms of its asymptotic growth) is $\un(n)$ as a bound of the actual corresponding extreme (leftmost) root of the $n$-th univariate Eulerian polynomial. This is the content of the next result.

\begin{proposicion}[Improving Corollary \ref{clearerasin} using the relaxation]\label{improvementcearer}
We have that $|q_{1}^{(n)}|\geq\un(n)$ with $\lim_{n\to\infty}\frac{\un(n)}{2^{n+1}}=1$.
\end{proposicion}

\begin{proof}
We write now the LMP defining the relaxation as \begin{gather*}
    \begin{pmatrix}
    n & 2^{n+1}-(n+2)\\2^{n+1}-(n+2) & 2 - 3^{1 + n}2 + 4^{1 + n} + n
    \end{pmatrix} + \\ x\begin{pmatrix}
    2^{n+1}-(n+2) &  2 - 3^{1 + n}2 + 4^{1 + n} + n\\ 2 - 3^{1 + n}2 + 4^{1 + n} + n & -2 - 2^{1 + n} 3^{2 + n} + 4^{1 + n}3 + 8^{1 + n} - n.
    \end{pmatrix}.
\end{gather*} Next, we calculate its determinant \begin{gather*}
    -4(-1 + 2^n)^2 + (-2 + 2^{2 + n} - 3^{1 + n}2 + 4^{1 + n})n
 +\\ x(-4(-1 + 2^n)(1 + 2^{1 + 2n} - 3^{1 + n}) \\+ 2(1 - 2^n + 2^{3 + 2n} + 2^{2 + 3n} - 3^{1 + n} - 2^n 3^{2 + n})n
)\\ + x^{2}(-2^{2 + n} - 2^{3 + 2n}5 + 3^{1 + n}8 + 6^{2 + n} + 8^{1 + n} - 9^{1 + n}4 +\\12^{1 + n} - 2(2^n + 2^{1 + 2n}5 + 2^{2 + 3n} - 3^{1 + n}2 - 2^n 3^{2 + n})n).
\end{gather*} Finally, writing this quadratic polynomial as $c+bx+ax^{2}$ as a polynomial in the variable $x$, we can easily determine its roots because it is a polynomial of degree $2$. We select its root (both roots are negative because we are using a relaxation that respects the rigidly convex set of the original polynomial) of smallest absolute value and call it $-\un(n)$. This root is our outer approximation to the largest root of the univariate Eulerian polynomial. In symbols, this root equals \begin{gather*}
   -\un(n)=\frac{-b+\sqrt{b^2-4ac}}{2a}.
\end{gather*} This is so because $\lim_{n\to\infty}\frac{a}{12^{n+1}}=1$ (so $a$ is positive for all $n$ big enough) and in $b$ the terms winning asymptotically are $$-2^{1+3n}4+8^{n+1}n=8^{1+n}n-8^{1+n}=(n-1)8^{1+n}>0$$ (so, for all $n$ big enough, $-b<0$) and therefore, as all the roots of these polynomials must be real, the term inside the square root $b^2-4ac\geq0$ so the root with smallest absolute value is the root in which the $+$ is taken before the square root. Knowing this, we can now use palindromicity of univariate Eulerian polynomials in order to see that \begin{gather}\label{uni}|q_{1}^{(n)}|\geq\frac{2a}{b-\sqrt{b^2-4ac}}:=\un(n).\end{gather} Now we just have to see the last part about the asymptotic growth of $\un(n)$, i.e., we need to check that $\lim_{n\to\infty}\frac{\un(n)}{2^{n+1}}=1$. This is clear\footnote{Be careful, some computer software is unable to compute this limit correctly because of the square root in the denominator that goes to $0$. It seems like the software gets lost before the step of multiplying by the conjugate of the expression is done and thus proceeds wrongly! This happened to us with Mathematica} because, when $n\to\infty$, \begin{gather*}
\frac{\frac{2a}{b-\sqrt{b^2-4ac}}}{2^{n+1}}=\frac{\frac{2a}{2^{n+1}}}{b-\sqrt{b^2-4ac}}=\frac{\frac{2a}{2^{n+1}}(b+\sqrt{b^2-4ac})}{b^2-(b^2-4ac)}=\\\frac{\frac{2a}{2^{n+1}}(b+\sqrt{b^2-4ac})}{4ac}\to 1.\end{gather*} We can check this observing that the term winning asymptotically in the last denominator is $2^{4 + 3 n} n$ and the term winning asymptotically in the numerator is $2^{3 + 3 n} n+\sqrt{2^{6 + 6 n} n^2}=2^{4 + 3 n} n.$ This confirms that the quotient tends to $1$ as $n\to\infty$.
\end{proof}

As a direct corollary, we see now that we can already answer Question \ref{mezo}. The proof is straightforward now using the sandwich of asymptotic inequalities provided by Proposition \ref{improvementcearer} and Corollary \ref{clearerasin}.

\begin{corolario}[Answer using the relaxation]
We have that $\lim_{n\to\infty}\frac{|q_{1}^{(n)}|}{2^{n+1}}=1$.
\end{corolario}

This gives us the best possible bound obtainable through the application of the relaxation to the univariate Eulerian polynomial. Now we question if we can go beyond this. But before we will see the other strategy in action with this same relaxation in order to compare and because the bound obtained this way will show us partially the behaviour of the multivariate relaxation in cases where we do not take a lot of care about the vectors we are using.

\section{Best univariate bound and bound using a vector}\label{bestuni}

We obtained above the best possible bound that the univariate relaxation provides. In the next Proposition \ref{anteserajemplo}, we will see that, for univariate Eulerian polynomials, the best possible bound given by the univariate relaxation and the one obtained through the strategy of checking through a vector have the same asymptotic growth $2^{n+1}$.

\begin{remark}[Identifying the improvement]
The bound obtained in the section above is the best one that the relaxation can provide when applied to the univariate Eulerian polynomials. It was possible to explicitly compute this optimal bound because the matrices were small enough in this case. However, in the future, we will deal with matrices too big to let us follow this same approach. For this reason we look now at approximate approaches through exercises of \textit{eigenvector guessing} as commented in the strategies exposed above. This method will therefore be of fundamental relevance in the future when we move to the fully multivariate realm.
\end{remark}

Thus, we saw that the best possible univariate bound provided by the relaxation can be written explicitly. We remark this fact showed in Equation \ref{uni}. These nice expressions would be too complicated to obtain for polynomials involving many more (or an increasing number of) variables. This happens because more variables imply that the definition of the relaxation requires dealing with LMPs of bigger size. The degree of the corresponding determinants will then grow accordingly. This growth obviously makes solving for the corresponding roots a practically (and also theoretically, symbolically and even numerically) undoable exercise. 

For this reason, we need to develop methods that can jump over the hard requirements of the degree growth of the determinant having in mind that our LMPs grow when the number of variables increases. Hence, instead of dealing with the determinant, we can fix, e.g., the simple sequence of vectors $v=\{(1,1)\}_{n=1}^{\infty}$ and use it to extract another sequence of bounds. We will see that such sequence of bounds has actually the same asymptotic growth $2^{n+1}$ as the sequence of absolute values of the actual corresponding leftmost roots of the sequence of univariate Eulerian polynomials.

\begin{proposicion}[Same asymptotics checking through the constant sequence of vectors $v=\{(1,1)\}_{n=1}^{\infty}$]\label{anteserajemplo}
The sequence $b_{v}\colon\mathbb{N}\to\mathbb{R}$ of inner bounds for the smallest root of the univariate Eulerian polynomial obtained through the guess of the generalized eigenvector of its relaxation given by the constant sequence of vectors $v=\{(1,1)\}_{n=1}^{\infty}$ verifies $\lim_{n\to\infty}\frac{b_{v}}{2^{n+1}}.$
\end{proposicion} 

\begin{proof}
Using the constant sequence of vectors $v=\{(1,1)\}_{n=1}^{\infty}$, we get the inequality \begin{gather*}L_{p}(1)+2L_{p}(x)+L_{p}(x^2)+x(L_{p}(x)+2L_{p}(x^2)+L_{p}(x^3))=\\-2 + 2^{2 + n} -3^{1 + n}2 + 4^{1 + n}+\\x(2^{1 + n} + 2^{3 + 2n} -  3^{1 + n}4 - 2^{1 + n} 3^{2 + n} + 4^{1 + n}3 + 8^{1 + n}
)\geq0.\end{gather*} Now, proceeding as we did previously, we get the bound \begin{gather*}
q_{1}^{(n)}\leq\frac{2^{1 + n} + 2^{3 + 2n} -  3^{1 + n}4 - 2^{1 + n} 3^{2 + n} + 4^{1 + n}3 + 8^{1 + n}}{2 - 2^{2 + n} + 3^{1 + n}2 - 4^{1 + n}}:=-b_{v}(n)\end{gather*} or, in absolute value and asymptotically, we obtain what we want because we can see that $|q_{1}^{(n)}|\geq b_{v}(n)\sim\big{(}\frac{8}{4}\big{)}^{n+1}=2^{n+1}$ as it is clear that we have \begin{gather*}\frac{2^{1 + n} + 2^{3 + 2n} -  3^{1 + n}4 - 2^{1 + n} 3^{2 + n} + 4^{1 + n}3 + 8^{1 + n}}{2^{n+1}(-2 + 2^{2 + n} - 3^{1 + n}2 + 4^{1 + n})}=\frac{b_{v}(n)}{2^{n+1}}\to1\end{gather*} when $n\to\infty$.
\end{proof}

This proof clearly shows us why it is beneficial to look through guesses of generalized eigenvectors instead of spending our energy trying to establish the optimal bound given by the relaxation. This fact and the asymptotic bound by the other side provided already by Mez\H{o} give together a clear first order asymptotic estimate for $|q_{1}^{(n)}|$ that also answers Mez\H{o}'s question with $\beta=0$ and $d=1.$ Thus we obtained the following result.

\begin{proposicion}[Answering the question of Mez\H{o} with actual better bounds and tight asymptotics]
$I(n)\geq|q_{1}^{(n)}|\geq\un(n)\geq b_{(1,1)}(n).$
\end{proposicion}

\begin{proof}
The only thing to remark is that $\un(n)\geq b_{(1,1)}(n)$. This happens because $\un$ is the optimal bound and therefore $b_{(1,1)}(n)$ is further from the actual root.
\end{proof}

The bound $b_{(1,1)}(n)$ will appear again in the future in the multivariate case if we are naive. We will see this in more detail soon but we can already advance that, if in the multivariate relaxation we guess the simple sequence of eigenvector $\mathbf{1}\in\mathbb{R}^{n}$ for each $n$, then the multivariate matrix compress to the univariate matrix and this process kills all the nuances introduced by the many more variables. In consequence, this choice gives back exactly the same bound $b_{(1,1)}(n)$ destroying all the information and advantage that having a growing number of variables can actually give to us. We will see this in more detail in future sections when we analyze \textit{the whole power of the relaxation} instead of just the surface level achieved here with the univariate relaxation. We care about using the \textit{whole power of the} multivariate relaxation because the main novelty of this approach lies actually in its multivariateness. The univariate relaxation has already appeared previously in the literature, as we will immediately see in the next section..

\section{Previous appearances of the univariate relaxation and its whole power}\label{previous}

The univariate relaxation has shown already to be powerful. It guided us towards establishing the first order growth answering Mez\H{o}'s question. But this univariate relaxation is not new. It appeared before in \cite[Equation 2.2.9]{szeg1939orthogonal}. For more details on this early appearance, we refer also to \cite{blekherman2020generalized}. This is why we are particularly interested in going to the multivariate case. Not just because we want to study and bound the rigidly convex sets of these multivariate Eulerian polynomial globally but also because we want to apply the multivariate relaxation in all its might. We want to do this because this approach is new and gives us a link and a connection to results coming from real algebraic geometry. We therefore expect that going multivariate will greatly improve our overview of these topics. We will indeed see why this happens soon.

Before that, we need a more detailed treatment of asymptotics in order to compare these bounds properly. We remark that we will use a particular form of representing and extracting asymptotic expansions in this setting. In particular, we will expand the asymptomatic growth through very simple functions. Our reference for asymptotic expansions is \cite{malham2005introduction}, but we will just use a very tiny and particular fragment of this theory.

\begin{remark}[Necessary tools in asymptotics to push beyond]\label{expsca}
Our \textit{asymptotic scale} will be the one given by functions of the form $a^{n}$ for $a\in\mathbb{R}_{\geq0}$. Moreover, we fix this scale all the time in the sense that, whenever we compute some term of the growth of a sequence in this scale, we understand that we do not ever change the scale when we extract more terms of the asymptotic expansion of that same sequence. Thus, our choice of the scale used is homogeneous, stable and invariable throughout this whole article.
\end{remark}

The relaxation has proven that it is capable of determining an inner bound having a tight asymptotic growth for the extreme roots of the univariate Eulerian polynomials. This makes us wonder how far this relaxation can reach. We can mention several interesting facts now.

\begin{remark}[The real power of the relaxation]
The main power of the relaxation comes from the fact that it is relatively cheap to compute. It only uses the degree three part of the polynomial at play. Moreover, when we deal with the multivariate extension, the relaxation not only bounds the original univariate polynomial, but the whole ovaloid of the multivariate extension, as we can see in \cite{main}. This uniformity means that we can extract knowledge not just about the original univariate polynomials but also about the other univariate polynomials that lie in restrictions close to this one, that is, around the diagonal.
\end{remark}

Around \cite[Equation 27]{macnameeii} one can find interesting information about methods to estimate roots of polynomials computing eigenvalues of companion matrices. These methods are therefore certainly similar to our method using the relaxation that we explore here. However, there, RZ-ness is not taken into account. Therefore a big chunk of the theory that helps us build the story of our method is missing and non-reconstructible from the approaches taken there. But not everything is bad: at least, we found a broad family of strategies where our method could fit. Finally, applying the relaxation to these well-known polynomials allows us to reveal new venues in the limits and extensions of the relaxation itself. All in all, the relaxation shows us here its versatility pointing towards many different directions. All these new directions surge from our early explorations about the many venues where this relaxation could be used or exploited in order to achieve further understanding of real-rooted (or RZ) polynomials and their features.

We want to make now sure that the reader completely understands the nuances of our approach. Making this clear will ensure a smoother trip through our future discussions and, therefore, a greater enjoyment along the ways we will follow next.

\begin{warning}[Ground necessary before going beyond]
What we did in this section can be considered an easy example of what is to come afterwards. Guessing the correct structure for the approximated eigenvectors, dealing with the huge expressions that will appear, extracting several terms of the asymptotics and computing optima will become more complicated due to the cumbersome length of some expressions. However, we saw already the main problems that will appear and how we will approach them in order to overcome them.
\end{warning}

One thing we have to notice immediately is how our terminology is not arbitrary. We are going to guess ``approximations'' of eigenvectors in a very particular sense that fulfils our objectives here. However, these guesses can still be very far from being actual good approximations to eigenvectors. We do not care about this and, for that reason, we talk about ``guesses''.

\begin{remark}[Guesses versus approximations of eigenvectors]
Notice that, although we talk frequently about ``approximated eigenvectors'', we do not have, in principle, any guarantee that these guesses \textit{approximate} the actual eigenvectors that we would like to know. These eigenvectors might have very complicated and computationally difficult to deal with expressions. We just know that our guesses are computationally easy to manipulate and work with while, at the same time, they get \textit{close enough} for our interests to provide good approximations to a generalized eigenvalue problem through linearization. This is just how we obtain our bounds. For these reasons, we will usually just refer to these as ``guesses'' instead of ``approximations''.
\end{remark}

In the next section, we show why looking at the univariate polynomials as injected in the diagonal of the multivariate ones gives us a good way to measure how good the (global ovaloid) bounds provided by the application of the relaxation become when $n$ goes to infinity, that is, asymptotically (and globally).

\section{A measure of accuracy through the diagonal}\label{ame}

One advantage of the use of the relaxation when analyzing the multivariate Eulerian polynomials is its globality. The relaxation gives a bound ovaloid for the ovaloid of the corresponding RZ polynomial in the correct dimension. However, as $n$ grows, so does the dimension and therefore the shape of the ovaloids being bounded by the relaxation changes is a drastical way (that is, even dimensionally). For this reason, it is not an easy task to measure how good the relaxation becomes asymptotically at bounding these ovaloids. Now, for measuring this asymptotic accuracy of the relaxation as a global bounding object, one can think of many different tools and strategies. Here we will be humble and look at a relatively straightforward (and cheap to compute) measure that connects us again to the univariate Eulerian polynomials. In particular, these polynomials are injected in the diagonal because we know that $A_{n}(x,\dots,x)=A_{n}(x).$ Thus measuring the accuracy of the relaxation when restricted to the diagonal amounts to measuring the accuracy of the relaxation to approximate the roots of the univariate Eulerian polynomials. Notice that choosing the diagonal as a restriction is also natural as it means that we look always, for each $n$, at the direction $\mathbf{1}\in\mathbb{R}^{n+1}$ no matter which dimension our ovaloid has in each case.

\begin{definicion}\label{linesres}
Let $v=\{v_{i}\}_{i=1}^{\infty}$ be a sequence of vectors with $v_{i}\in\mathbb{R}^{i}$ and $B=\{B_{i}\}_{i=1}^{\infty}$ a sequence of polynomials with $B_{i}\in\mathbb{R}[x_{1},\dots,x_{i}]$. We call the sequence of univariate polynomials $B_{v}(t)=\{B_{i}(tv_{i})\}_{i=1}^{\infty}$ the \textit{univariate selection of $B$ through $v$}. If there exists a function $g\colon\mathbb{N}\to\mathbb{R}$ such that $\lim_{n\to\infty}\frac{B_{i}(tv_{i})}{g(i)}=1$ we say that such sequence $g$ \textit{has the same asymptotic growth as $B$ through $v$}.
\end{definicion}

Thus, as we can easily see, investigating the asymptotic growth along the diagonal is a particular case of this. In the diagonal we find the Eulerian polynomials.

\begin{remark}[Line injected measures of accuracy]
We center our attention to the diagonal because there we find the already well studied Eulerian polynomials. But, as we saw in the Definition \ref {linesres} above, we could have studied this accuracy through other choices of directions. However, doing this would not be as effective because focusing on univariate Eulerian polynomials allows us to use many facts and root approximations that are already well-known for these polynomials. This knowledge means then that we can compare now what the relaxation does along this line with what other univariate methods do. Therefore, looking through the diagonal, we obtain a way to compare the power of the use of the relaxation to the use of other methods already presented in the literature about these polynomials.
\end{remark}

Hence observe that Definition \ref{linesres} gives us the tools for measuring (global) accuracy that we wanted. It is possibly one of the simplest tools for measuring global accuracy of multivariate (ovaloid) bounds. But it has the powerful advantage that it clearly connects our multivariate theory immediately to the well-known and already developed univariate one. Now we can direclty work with the relaxation along the diagonal in order to compare the effects of our method based on the multivariate relaxation with the other methods we saw above. Remember that the use of the univariate relaxation was already the winner among these methods. We will soon see that going multivariate gives even better bounds than the univariate relaxation, establishing therefore an edge in (global) bound computing. First, we have to get used to work with the relaxation in a correct way.

\section{Working with a relaxation}\label{wo}

Recalling what we have done, we draw the following path. We started the journey applying the relaxation to univariate polynomials. Doing this already provides a very good bound in asymptotic terms that immediately matches the first asymptotic growth term of the root we are bounding in the natural exponential scale we are using for asymptotic expansions.

\begin{remark}[Univariate bound already good]
The univariate bound computed sections before grows like $2^{n+1}+o(2^{n})$. The actual extreme root of the univariate Eulerian polynomial $A_{n}$ has the same asymptotic growth. For this reason, we have to search for improvements down further terms in the asymptotic expansion (that is, in terms of the form $a^n$ with $0<a<2$). We do not yet know how far in this asymptotic expansion (that is, how small $a$ should be) we have to go. This is why we will compare bounds here through the use of differences. 
\end{remark}

The size of the LMP defining our relaxation depends on the number of variables. Thus any hope to actually improve the bound provided by the relaxation has to be connected to a mechanism letting us increase the number of variables as the degree grows. In order to accomplish this, we use multivariate Eulerian polynomials in our next step.

\begin{remark}[The importance of multivariability]
Multivariate Eulerian polynomials have relaxations that grow with the degree because the corresponding number of variables also increases accordingly. In this sense, they are a step in the correct direction, but clearly not the end of the story because we know that an even more explosive growth would be needed if we wanted to capture more information about the multivariate Eulerian polynomials. It is clear that the relaxation needs additional variables to have the chance to improve beyond this.
\end{remark}

Fortunately, the multivariate Eulerian polynomials were already studied before in \cite{visontai2013stable,haglund2012stable}. This previous research gives us an advantage now because it allows us to look immediately into these polynomials and their roots knowing already many important facts about them. We already saw the main picture of our procedure. Let us now concretize it.

\begin{procedimiento}
In particular, we will use mainly the fact that these polynomials are RZ, which allows us to build the relaxation as in Section \ref{ext}. Thus, using Procedure \ref{procesobound}, we obtain in general a lower bound $z_{n}\leq x_{n,r}$ that will depend on our linearization, that is, on the vector $v_{n}$ that we choose for each $n$. Remembering that $x_{n,r}\leq q_{n}^{(n)}$, this gives a (linearized) lower bound for $x_{n,r}$ that, ultimately, provides a bound for the root $q_{n}^{(n)}$ that is much easier to compute than $x_{n,r}$ itself. As we want to look at the root having largest absolute value, we can therefore use the palindromicity of Eulerian polynomials in the inequality $z_{n}\leq q_{n}^{(n)}$ to produce the bound $q_{1}^{(n)}=\frac{1}{q_{n}^{(n)}}\leq\frac{1}{z_{n}}:=w_{n}.$
\end{procedimiento}

This is how we will work with the relaxation. Now that we really go multivariate, we have to be cautious before using the relaxation in this multivariate setting. We learn more about this in the next section.

\section{Choosing vectors with care}\label{ch}

The number of variables will grow now in each iteration. This effect will also increase the size of our LMPs. The benefit of this will be evident soon. But for this benefit to become clear we will have to start guessing correctly our approximations to the corresponding generalized eigenvector problem. This is how we compute approximations to the actual optimal bound provided by the relaxation without having to deal with computing roots of a determinantal polynomial, which would be highly impractical now.

\begin{remark}[Guessing the eigenvector]
The only approach that is practical and doable for us in this case is guessing the eigenvector because the growing size of the LMPs in the sequence only let us this option available. Still, finding a correct guess will be a tricky task.
\end{remark}

For this task, we will only use simple techniques. We do this because we prefer not to complicate our setting even more in this introductory paper about our technique. But the reader can easily see that there is a wide path to exploit and explore here in this direction.

\begin{observacion}[Advanced guessing]
Notice that here we will avoid using very advanced techniques or tools for guessing these eigenvectors in order to keep things simple enough, but it is a direction worth exploring in the future. For some of these techniques tailored to guess or approximate eigenvectors we direct the interested reader to \cite{rump2001computational,denton2022eigenvectors,karpfinger2022numerical,garcia2024finding}.
\end{observacion}

In order to use the whole power of the multivariate relaxation we have to be cautious about our choice of vectors. In particular, we must ensure that we do not accidentally come back to the univariate setting through a bad choice. We show how this could happen. For this, we introduce now a weak sense of compression of matrices. This concept will help us to understand the origins of our problem here when we do not take enough care when choosing the vector we use to linearize in Procedure \ref{procesobound}. For further and more insightful concepts and methods of matrix compression and in order to understand better the importance of such concept, we direct the interested reader to \cite{cheng2005compression,oppelstrup2013matrix,paixao2015matrix,nemtsov2016matrix,elgohary2018compressed,jindal2021compression,belton2022matrix,kadowaki2022lossy,martel2022compressed,saha2023matrix,levitt2024linear,levitt2024randomized}.

\begin{definicion}[Matrix row and column operations as bilinear form compressions via vectors]\label{compressionmat}
Let $R$ be a ring and $M\in R^{n\times m}$ a matrix over that ring. We say that the matrix $N\in R^{n'\times m'}$ with $n'\leq n$ and $m'\leq m$ is a \textit{compression} of $M$ if $N$ is obtained from $M$ after summing some of its rows and columns suppressing the appearance of the rows and columns being added to another one.
\end{definicion}

In view of this, we see immediately that the all-ones vector $(\mathbf{1})$ can never achieve anything better than the univariate relaxation. This happens because it is clear that using this vector has the same effect as compressing the multivariate LMP to the univariate one we saw before.

\begin{observacion}[Action of the sequence of all-ones vectors]
Choosing the sequence of vectors $\{(a,\mathbf{1})\in\mathbb{R}^{n+1}\}_{n=1}^{\infty}$ to linearize the multivariate LMP equals to choosing the constant sequence of vectors $\{(a,1)\in\mathbb{R}^{2}\}_{n=1}^{\infty}$ to linearize the univariate LMP. Thus, in this sense, using a sequence of vectors of the form $\{(a,\mathbf{1})\in\mathbb{R}^{n+1}\}_{n=1}^{\infty}$ to linearize the multivariate LMP gives back the LMP of the univariate relaxation through an immediate compression of the last $n$ rows and columns and therefore does not allow us to obtain any improvement with respect to the application of the relaxation to the univariate Eulerian polynomial.
\end{observacion}

In this way, we see that we have to be careful with our choice of guesses of vectors in order to avoid ending up exactly at the same point as we were before. We have to avoid, in particular, falling back into the univariate setting without realizing.

\section{Multivariate relaxation really at play}\label{multrel}

When we try to guess a sequence of vectors to use here we see that the most natural choices fail to provide an improvement in the bound obtained via the univariate relaxation. For this reason, in order to make good guesses, we need to perform some numerical experiments that allow us to understand better the most adequate choices of vectors to linearize our problem. This further understanding will help us to surmise the form of a suitable sequence of vectors that we will use to obtain the best bound in this article. The vectors in our sequence will have easy expressions. This will facilitate our work with them. The bound obtained in this way will, in particular, beat the previous univariate bound. This new bound will show that the multivariate relaxation allows us to obtain better bounds. Hence, through this discovery, we pave the road for further explorations about other possible linearizations using more effective families of vectors.

Therefore, all in all, our objective in this section consists in finding the way of choosing a sequence of vectors that linearizes the relaxation without putting us back in the univariate case and producing an improvement in the bound obtained. This exercise will allow us to show that the use of the multivariate Eulerian polynomials let us increase the accuracy of the relaxation in the diagonal, where we find the univariate Eulerian polynomials.

\begin{objetivo}[Addition of new variable information]
Through numerical experiments, we know that the application of the relaxation to the multivariate Eulerian polynomials increases the accuracy of our bounds for the extreme roots of the univariate Eulerian polynomials. Finding how to write down a suitable sequence of vectors that allows us to effectively prove this through practically doable computations is our task here.
\end{objetivo}

In order to do this, we will have to be more precise in the asymptotics so that we can extract the correct structure for the sequence of approximated generalized eigenvectors we need. Thus, in order to be able to estimate correctly the improvements that these choices of sequences of approximated generalized eigenvectors will provide, we have to improve our management of deeper asymptotics. For this reason, we explain here how we are going to analyze the information about the asymptotic growth of the sequences that we will soon compute.

\begin{remark}[Asymptotic tools to measure improvements]
Deeper terms of the asymptotic growth of the sequences at play will be important. In particular, we will extensively use the exponential scale introduced in Remark \ref{expsca}. We will deal with (in-)equations in the corresponding coefficients of the expansions of the bounds in these scales.
\end{remark}

Here we will experience the real power of the relaxation. But we will also discover at the same time the complications attached to the successful management of the multivariate relaxation. For this reason we need a strong justification to go multivariate.

\begin{objetivo}[Multivariateness helps]
We have to find ways to use the additional variables in our polynomials in order to improve the bounds. This objective makes sense because we run numerical experiments that tell us that the relaxation applied to the multivariate Eulerian polynomials should provide us with better results. We therefore want to actually prove that we do obtain something better using this method.
\end{objetivo}

Recall that, as $1$ can never be a top, some difficulties in the notation appear. We are referring here to the emergence of ghost variables.

\begin{reminder}[Ghost variables, indexing and size]
Recall that we decided to write the $n$-th Eulerian polynomial as a polynomial in the $n$ variables $x_{2},\dots,x_{n+1}$ because $x_{1}$ is never appearing. Hence the relaxation corresponding to this $n$-th multivariate Eulerian polynomial has size $n+2$. However the second row, column and matrix coefficient of the corresponding relaxation are always just $0$. Thus these are irrelevant. We keep them here just for indexing purposes. We can keep them because we are going to linearize and therefore they do not bother our computations contrary to what happened if we tried to compute the corresponding determinant with these unfortunate null structures around.
\end{reminder}

We follow here a path guided by numerical experiments. Thus is how we obtain the insights that we exploit here.

\begin{observacion}[Clearing the path with experiments]
Numerical experiments show that the multivariate relaxation gives, in fact, better bounds. This is what pushed us in this direction. As a consequence, we will explore further the behavior of the relaxation under linearization through sequences of vectors with a deeper structure. We do this with the hope that these structures could allow us to extract more information about the polynomial from the relaxation through the use of more information about the variable splitting performed in the polynomial by the refined count of descents collecting the corresponding tops as variable indices (tags).
\end{observacion}

Now we analyze what is achievable. We will see that some paths are either plainly impossible or computationally unfeasible.

\section{Limitations of any possible improvement}\label{limitations}

We know that we cannot improve the first order asymptotics because that order is already tight in the exponential scale we use. But a good enough choice of vectors in the multivariate case could still improve the extreme root estimates in terms of higher order asymptotic terms. For this, choosing a good sequence of vectors is fundamental. In other case, this task becomes unfeasible. We are more clear about this in the next two remarks about how bounds are estimated using either the determinant or linearizing vectors. We start pointing out to the limitations of trying to use the determinant of the relaxation.

\begin{remark}[Through the determinant]
Using the determinant of the relaxation in order to produce a bound for its innermost root is not a useful task. Managing the polynomial emeging from such determinant is not easier than managing the original univariate Eulerian polynomial. In fact, not only the coefficients are way more complicated, but also some nice symmetries are lost. Additionally, there are many extremely cumbersome computations in the middle that, overall, lead to an unreasonably long task in order to obtain a polynomial of around the same degree as the polynomial we started with.
\end{remark}

Clearly, directly dealing with the determinant is therefore not the path because the determinant might be more accurate, but it is fundamental to balance accuracy with workability. We get this balance linearizing through a suitable sequence of vectors.

Because of the form of the sequence of vectors we find, we will have to deal with a further complication: at some point, we will have to optimize over an entry in the vector in order to achieve an optimal bound through the corresponding linearization. However, doing this optimization is fortunately an achievable task from a computational and practical point of view.

The first sequence of vectors that we could find beating the univariate application is, in some sense, very simple except for the mentioned necessary optimization step. We describe next the general form of this sequence before the optimization step.

\begin{remark}[Through a new sequence of vectors]
Our experiments show that one good choice of sequence of vectors to linearize the relaxation has the form $\{(y,0,1,-\mathbf{1})\in\mathbb{R}^{n+2}\}_{n=1}^{\infty},$ where $y$ is a sequence of reals (depending on $n$) that has to be determined through a optimization step and instead of $0$ we could put any other number without this affecting the result because that entry corresponds to the mentioned ghost variable $x_{1}$ since $1$ can never be a top.
\end{remark}

The optimization of the first entry $y$ will be a fundamental part of our job. The sequence of vectors $\{(y,0,1,-\mathbf{1})\in\mathbb{R}^{n+2}\}_{n=1}^{\infty}$ finally breaks the univariateness. Now we are really using more than one variable. Thus, in this case and for the first time in this bounding process of the extreme roots of the univariate Eulerian polynomials through the relaxation, we use a deeper result from the theory of real algebraic geometry.

\begin{observacion}[Using Helton-Vinnikov for the first time]
Schweighofer's proof in \cite{main} that the relaxation is indeed a relaxation for RCSs of RZ polynomials of strictly more than one variable uses Helton-Vinnikov theorem. We did not use this theorem so far because the winner guess for a suitable sequence of vectors providing a bound up to this point was the all-ones vector $\mathbf{1}$. As we commented sections above, the use of this vector is equivalent (through the compression concept introduced in Definition \ref{compressionmat}) to considering only the relaxation applied to the univariate Eulerian polynomial. The fact that the relaxation is indeed a relaxation in the univariate setting is much easier to prove and, therefore, this fact was already shown in \cite{szeg1939orthogonal}, as we mentioned.
\end{observacion}

We are now in the right path. We just have to optimize to obtain the correct $y$.

\section{Preparation for the computation, optimization and comparison}\label{preparation}

Observe that optimizing the bound provided by the sequence of vectors $v:=\{(y,0,1,-\mathbf{1})\in\mathbb{R}^{n+2}\}_{n=1}^{\infty}$ over the simple parameter $y$ will give us the entries of the corresponding sequence reals given by $y$ as a function of $n$. For this, we will compute first the bound obtained through such method in terms of $y$ as a simple parameter. After that, we will maximize the absolute value of our bound for $y\in\mathbb{R}$. This will allow us to obtain a explicit expression for the sequence of reals $y$ in terms of $n$. Plugging this sequence of reals obtained in this way for $y$ in the sequence of formal bounds instead of the parameter $y$ we obtain an actual sequence of bounds depending only on the indexing parameter of the sequence $n$. Thus, we can eventually compare this multivariate sequence of bounds with the univariate sequence of bounds obtained some sections above. We will do this considering the sequence of differences. This process will end up giving the following result whose proof is expanded in the next sections.

\begin{proposicion}[Beating univariate]\label{multundiff}
Call $\un,\mult_{v}\colon\mathbb{N}\to\mathbb{R}_{>0}$ the sequence of bounds for the absolute value of the smallest (leftmost) root of the univariate Eulerian polynomial $A_{n}$ obtained after applying the relaxation to the univariate $A_{n}$ and the multivariate $A_{n}$ Eulerian polynomials respectively, where $\mult_{v}$ is obtained after linearizing using the vector $v$ and $\un$ is the optimal bound given by the actual root of the determinant of the relaxation. Remember also that $v$ depends on $y$ and in $n$ and that the optimal $y$ will depend on $n$. Then there exists a sequence $y\colon\mathbb{N}\to\mathbb{R}$ such that, for $n$ big enough, the difference $\mult_{v}(n)-\un(n)\sim\frac{1}{2}\left(\frac{3}{4}\right)^{n}$ is positive.
\end{proposicion}

Proposition \ref{multundiff} confirms us that we have in fact an improvement with respect to the univariate approach. The positivity of the difference $\mult_{v}-\un$ at each $n$ certifies that the new linearization using the new sequence of vectors $v:=\{(y,0,1,-\mathbf{1})\in\mathbb{R}^{n+2}\}_{n=1}^{\infty}$ ends up being closer to the actual extreme root of the corresponding univariate Eulerian polynomial $A_{n}$. Thus we obtain immediately the next corollary.

\begin{corolario}
The application of the relaxation to the multivariate Eulerian polynomials produce better bounds for the extreme roots of the univariate Eulerian polynomial than the application of the relaxation to the univariate Eulerian polynomials.
\end{corolario}

The rest of this article is the proof of Proposition \ref{multundiff}. We break this task into several sections for more clarity. First, we have to compute the form of the bound from the corresponding $L$-forms.

\section{Form of the bound obtained through $(y,0,1,-\mathbf{1})\in\mathbb{R}^{n+2}$}\label{formof}

We start by computing the inequality that gives the bound. The bound obtained through the choice of the sequence of vectors $\{(y,0,1,-\mathbf{1})\in\mathbb{R}^{n+2}\}_{n=1}^{\infty}$ comes from an inequation of the form $D+xN\geq0$ with with $D=$ \begin{gather*}y(yL_{p}(1)+L_{p}(x_{2})-\sum_{i=3}^{n+1}L_{p}(x_{i}))+yL_{p}(x_{2})+L_{p}(x_{2}^{2})-\sum_{i=3}^{n+1}L_{p}(x_{2}x_{i})\\-\sum_{j=3}^{n+1}\left(yL_{p}(x_{j})+L_{p}(x_{j}x_{2})-L_{p}(x_{j}^{2})-\sum_{j\neq i=3}L_{p}(x_{i}x_{j})\right)
\end{gather*} and $N=y(y\sum_{k=2}^{n+1}L_{p}(x_{k})+\sum_{k=2}^{n+1}L_{p}(x_{2}x_{k})-\sum_{i=3}^{n+1}\sum_{k=2}^{n+1}L_{p}(x_{k}x_{i}))+$ \begin{gather*}
y\sum_{k=2}^{n+1}L_{p}(x_{k}x_{2})+\sum_{k=2}^{n+1}L_{p}(x_{k}x_{2}^{2})-\sum_{i=3}^{n+1}\sum_{k=2}^{n+1}L_{p}(x_{k}x_{2}x_{i})\\-\sum_{j=3}^{n+1}\left(y\sum_{k=2}^{n+1}L_{p}(x_{k}x_{j})+\sum_{k=2}^{n+1}L_{p}(x_{k}x_{2}x_{j})-\sum_{i=3}^{n+1}\sum_{k=2}^{n+1}L_{p}(x_{k}x_{j}x_{i})\right)=\\y^2\sum_{k=2}^{n+1}L_{p}(x_{k})+2y\sum_{k=2}^{n+1}L_{p}(x_{2}x_{k})-2y\sum_{i=3}^{n+1}\sum_{k=2}^{n+1}L_{p}(x_{k}x_{i})+\\
\left(\sum_{k=2}^{n+1}L_{p}(x_{k}x_{2}^{2})-\sum_{i=3}^{n+1}\sum_{k=2}^{n+1}L_{p}(x_{k}x_{2}x_{i})\right)\\+\left(-\sum_{i=3}^{n+1}\sum_{k=2}^{n+1}L_{p}(x_{k}x_{2}x_{i})+\sum_{j=3}^{n+1}\sum_{i=3}^{n+1}\sum_{k=2}^{n+1}L_{p}(x_{k}x_{j}x_{i})\right)=\\y^2\sum_{k=2}^{n+1}L_{p}(x_{k})+2yL_{p}(x_{2}^2)\\-2y\sum_{i=3}^{n+1}\sum_{k=3}^{n+1}L_{p}(x_{k}x_{i})+L_{p}(x_{2}^{3})-\sum_{i=3}^{n+1}\sum_{k=3}^{n+1}L_{p}(x_{k}x_{2}x_{i})\\-\sum_{i=3}^{n+1}L_{p}(x_{2}^{2}x_{i})+\sum_{j=3}^{n+1}\sum_{i=3}^{n+1}\sum_{k=3}^{n+1}L_{p}(x_{k}x_{j}x_{i}).\end{gather*} Therefore, computing these last expressions, we obtain the following lemma telling us the form of the bound obtained in this way.

\begin{lema}[Form of the bound]\label{formofboundfirstwinner}
The bound for the extreme root of the $n$-th univariate Eulerian polynomial obtained through linearization of the relaxation applied to the corresponding multivariate Eulerian polynomial by the sequence of vectors $\{(y,0,1,-\mathbf{1})\in\mathbb{R}^{n+2}\}_{n=1}^{\infty}$ is of the form $q_{n}^{(n)}\geq-\frac{D}{N}$ with $D=$ \begin{gather*}
   10 - 2^{2 + n} + 2^{2 + 2n} - 2 \cdot 3^{1 + n} + n + 4y - 2^{1 + n}y + 
 n y + y(4 - 2^{1 + n} + n + ny)
\end{gather*} and $N=$ \begin{gather*}
    -10 + 2^{3 + n} - \frac{1}{3}2^{3 + 2n} - \frac{1}{3}2^{4 + 2n} + \frac{1}{7}2^{4 + 3n} + \frac{1}{7}2^{5 + 3n} +\\ 2 \cdot 3^n - 4 \cdot 3^{1 + n} + 2 \cdot 3^{2 + n} - \frac{1}{5}2^{1 + n}3^{3 + n} - 4^{1 + n} + 4^{2 + n} - \frac{6^{2 + n}}{5} +\\ \frac{8^{1 + n}}{7} - n - 8y - 2^{2 + n}y + 2^{3 + n}y - \frac{1}{3}2^{3 + 2n}y - \frac{1}{3}2^{4 + 2n}y +\\ 4 \cdot 3^{1 + n}y - 2ny - 2y^2 + 2^{1 + n}y^2 - ny^2.
\end{gather*}
\end{lema}

\begin{proof}
The proof consists in performing the last sums of the computations opening this subsection.
\end{proof}

While making these computations, some care has to be taken in order to avoid unnecessary mistakes and misunderstandings. For clarity in case the reader tries to check our arguments in depth, we comment on some complications that emerge while performing the iterated sums appearing in the expressions of $N$ and $D$ in terms of the $L$-form and therefore involved in the computations leading to the proof above.

\begin{warning}[On the complications of doing the sums]
In the computation of $D$ the possible problem is really minimal. This is so because during this computation the $L$-form only needs to be applied to quadratic polynomials. Still, we have to remember that the expression of $L_{p}(x_{i}^{2})$ is completely different from the expression of $L_{p}(x_{i}x_{j})$ for $i\neq j$. The same problem appears in the computation of $N$. But these problems even exacerbate there because of the emergence of application to cubic polynomials. Remember, in this sense, that our possible cases of confusion (and therefore necessary distinction) within the computations increase since, e.g., for $i\neq j$, the (formula for the) evaluation of the $L$-form given by $L_{p}(x_{i}^{2}x_{j})$ is drastically different depending on whether $i<j$ or $i>j$. Thus, we make the point here that following and correctly performing these (sums) computations is a laborious, obfuscate, repetitive and time-consuming task. In particular, this annoying task has to be done, in a great part, manually since the current symbolic computer software does not recognize the possibility (in fact doable with some work, as we show and point out here) of writing down the evaluation (over our sequence of vectors chosen here) of the quadratic form (associated with the sequence of relaxations of the RCSs of the sequence of multivariate Eulerian polynomials) we are performing in a closed form in elementary terms.
\end{warning}

Now we have to optimize in order to find the best $y$ in terms of $n$. This optimization of the bound will give us $y$ as function of $n$. This function $y$ will be obtained as a solution of a quadratic equation. This means that there will be square roots in our expressions. We will have to learn how to deal with these square roots correctly. We first see the optimization step in the next section.

\section{Optimization of the $y$ in terms of $n$}\label{optimization}

Now we have to determine $y$ wisely using that the bound that we have just obtained has the form $q_{n}^{n}\geq-\frac{D}{N}$. This imposes a first condition on this $y$. Namely, $y$ (which depends only on the parameter $n$) should make $N>0$ for $n$ big enough (so that we can isolate $x$ with an inequality in the correct direction). This implies, in particular, that it is enough if \begin{equation}\label{conditionsy} y>1 \mbox{\ and\ } \lim_{n\to\infty} \frac{\max\{2^{n+1}y^2, 2y4^{n+1}\}}{8^{n+1}}=0.\end{equation} We see this.

\begin{proposicion}
If Condition \ref{conditionsy} is true, then $N>0$ for all $n\in\mathbb{N}$ big enough.
\end{proposicion}

\begin{proof}
As $y>1$, the sum of terms that determine the first asymptotic growth of $N$ is $2^{n+1}y^2-\frac{1}{3}2^{3+2n}y-\frac{2}{3}2^{3+2n}y+\frac{2}{7}2^{3+3n}+\frac{4}{7}2^{3+3n}+\frac{1}{7}2^{3+3n}=2^{n+1}y^2-2^{3+2n}y+2^{3+3n}=2^{n+1}y^2-2y4^{1+n}+8^{1+n}.$ Now taking limits does the rest and finished the proof of this claim.
\end{proof}

Notice that $D$ must always be nonnegative for any $y$ because the relaxation has its initial matrix PSD. Hence we do not have to be as careful with $D$ as we have been with $N$ at the choice of $y$. For $D$, we just have to check that, for the chosen $y$, we have $D\neq0$. Fortunately, checking this condition about $D$ is quite easy using discriminants.

\begin{proposicion}
We have $D>0$ for all $n\in\mathbb{N}$ big enough and $y>0.$
\end{proposicion}

\begin{proof}
We see this fast using the fact that the discriminant of $D$ verifies $\Delta=64 - 2^{6 + n} + 2^{4 + 2n} - 8n - 2^{4 + 2n}n + 8 \cdot 3^{1 + n}n=64(1 - 2^{n}) - 8n + 16(1-n)4^{n} + 24 n 3^{n},$ which is clearly negative when $n$ is big enough and moreover the quadratic term in $D$ as a function of $y>1$ is $ny^{2}\neq0$. All this ensures that $D$ is a strictly quadratic polynomials with no real zeros and therefore strictly positive for whichever $y>1$ we might end up choosing, which proves the claim.
\end{proof}

This solves completely our first problem. Continuing with the task of computing the optimal choice of $y$, we rearrange the expression to remind that, as a bound in absolute values, our bound translates to $|q^{(n)}_{n}|\leq\frac{D}{N}$ (with $N>0$ for $n$ big enough). This, in turn, gives the bound $|q^{(n)}_{1}|\geq\frac{N}{D}$ for the absolute value of the leftmost root $q^{(n)}_{1}$ of the $n$-th univariate Eulerian polynomial. Notice that we used the palindromicity of these polynomials.

\begin{reminder}
As the univariate Eulerian polynomials are palindromic, if $r$ is a root of one of them, so is its reciprocral $\frac{1}{r}.$
\end{reminder}

Furthermore, as we are pursuing a better bound, we want $\un(n)<\mult_{v}(n),$ where $\un(n)$ is the sequence in $n$ found above giving the optimal bound (actual root of its corresponding determinant) obtained through the application of the relaxation to the univariate polynomial and $\mult_{v}(n)$ is the sequence in $n$ giving the (better multivariate) bound we are searching for now, which comes after linearizing through the sequence of vectors $v$ (depending on the sequence of reals $y\colon\mathbb{N}\to\mathbb{R}$) the sequence of LMIs obtained after applying the relaxation to the sequence of multivariate Eulerian polynomials.

\begin{reminder}
Therefore, we need to notice that $\mult_{v}$ is not even the best bound possible. It is just the bound we can work with here. We are linearizing through the sequence of vectors $v$ and not computing the actual roots of the sequence of determinants of the LMPs giving the sequence of LMIs defining the sequence of relaxations of the sequence of multivariate Eulerian polynomials. Contrary to that, $\un$ corresponds to the actual root of the corresponding determinant of the LMP giving the spectrahedron defining the relaxation of the univariate Eulerian polynomial. Thus $\mult_{v}$ is itself just a bound for how good $\mult$ (the one obtained through the root of the determinant of the relaxation applied to the multivariate polynomial) can be.
\end{reminder}

This last condition pushes us to study the sequence of differences $\{\mult_{v}(n)-\un(n)\}_{n=1}^{\infty}$ in order to find a sequence of reals $y$ (verifying the condition expressed in Equation \ref{conditionsy} above) such that, for all big enough $n$, the terms $\mult_{v}(n)-\un(n)>0$ are positive. In order to find our sequence of optimal $y$, as traditionally, we compute the derivative of the expression $\mult_{v}(n)$ viewed as a function of the parameter $y$ and search for its zeros. We do this for each $n$ now.

\begin{computacion}
In this case, writing $\mult_{v}(n)=\frac{N}{D}$ so the derivative is $\frac{N'D-ND'}{D^2}$, we are then looking for our optimal $y$ among the roots of the polynomial in $y$ given by the numerator $N'D-ND'$, which is quadratic in $y$ and depends on the parameter $n$. From the two roots obtained as usual via the well-known quadratic formula $y=\frac{-b\pm\sqrt{b^{2}-4ac}}{2a}$ we now have to find the way to determine which one gives the maximum of $\mult_{v}(n).$ To see this, observe that we have already seen that the discriminant of the denominator $D$ verifies $\Delta=64 - 2^{6 + n} + 2^{4 + 2n} - 8n - 2^{4 + 2n}n + 8 \cdot 3^{1 + n}n<0$ once $n$ is big enough (which was expectable beforehand also because the fraction $\frac{N}{D}$ cannot go to infinity as a function of $y$ because it is bounded by the absolute value of the extreme root of the $n$-th Eulerian polynomial for any given $n$) so it has no real zeros and therefore we can easily deduce now that the roots of $N'D-ND'$ correspond to a maximum and a minimum of $\mult(n)=\frac{N}{D}$ and that such function is continuous as a function of $y$. It might seem worth studying the horizontal asymptotes of such function but observe that $\lim_{y\to\infty}\frac{N}{D}=\frac{-2 + 2^{1 + n} - n}{n}=\lim_{y\to-\infty}\frac{N}{D},$ which does not directly help in our decision. However, now it is evident that, as a function of $y$, $\mult_{v}(n)$ cut the horizontal asymptote $f(y)=\frac{-2 + 2^{1 + n} - n}{n}$ in exactly one point and therefore the difference with this horizontal asymptote in any of the limits determines the relative position of the maximum and the minimum, as there are just two options on how this function looks. This observation finally puts us in the correct way to distinguish between the maximum and the minimum. Observe that the difference between the function $\mult_{v}(n)$ and the asymptote equals \begin{gather*}
    \mult_{v}(n) - \frac{-2 + 2^{1 + n} - n}{n}= \frac{z}{w}
\end{gather*} with $z=20 - 5 \cdot 2^{2 + n} - 2^{3 + n} + 2^{4 + 2n} - 2^{3 + 3n} - 4 \cdot 3^{1 + n} + 2^{2 + n} \cdot 3^{1 + n} + 2n - 2^{1 + n} n - 2^{2 + n} n + 2^{3 + n} n + 2^{2 + 2n} n - \frac{1}{3} \cdot 2^{3 + 2n} n - \frac{1}{3} \cdot 2^{4 + 2n} n + \frac{1}{7} \cdot 2^{4 + 3n} n + \frac{1}{7} \cdot 2^{5 + 3n} n + 2 \cdot 3^n n - \frac{1}{5} \cdot 2^{1 + n} \cdot 3^{3 + n} n - 4^{1 + n} n + 4^{2 + n} n - \frac{1}{5} \cdot 6^{2 + n} n + \frac{1}{7} \cdot 8^{1 + n} n + 16y - 2^{3 + n} y - 2^{4 + n} y + 2^{3 + 2n} y + 4ny - 2^{2 + n} ny - \frac{1}{3} \cdot 2^{3 + 2n} ny - \frac{1}{3} \cdot 2^{4 + 2n} ny + 4 \cdot 3^{1 + n} ny$ and $w=10n - 2^{2 + n} n + 2^{2 + 2n} n - 2 \cdot 3^{1 + n} n + n^2 + 8ny - 2^{2 + n} ny + 2n^2 y + n^2 y^2$ so the denominator $w$ is always positive at infinity while the numerator $z$, when $y$ goes to infinity, is dominated by the term multiplying $y$ given by $16 - 2^{3 + n} - 2^{4 + n} + 2^{3 + 2n} + 4n - 2^{2 + n}n - \frac{1}{3} \cdot 2^{3 + 2n}n - \frac{1}{3} \cdot 2^{4 + 2n}n + 4 \cdot 3^{1 + n}n
,$ which is dominated by the terms $- \frac{1}{3} \cdot 2^{3 + 2n}n - \frac{1}{3} \cdot 2^{4 + 2n}n=-2^{3 + 2n}n<0$ so the $\lim_{y\to-\infty}\frac{z}{w}=0^{+}$ (because $z\approx-2^{3 + 2n}ny>0$ as $y\to-\infty$) which implies that, for $y$ close enough to $-\infty$, $\mult_{v}(n)-\frac{-2 + 2^{1 + n} - n}{n}>0$ or, equivalently, $\mult_{v}(n)>\frac{-2 + 2^{1 + n} - n}{n}$ and, therefore, the function $\mult_{v}(n)$, when viewed from left to right, increases from above the asymptote until it reaches a maximum over it and then decreases towards a minimum reached under the asymptote after crossing it before increasing back towards the asymptote approaching it from below. All in all, we have established that the maximum is reached at the leftmost of the two zeros of $N'D-ND'.$
\end{computacion}

In this way, we have determined which root of $N'D-ND'$ corresponds to a maximum of the fraction $\frac{N}{D}$ and now we know exactly the optimal choice of the sequence of reals $y$ in our way towards the best bound attainable through the sequence of vectors $\{(y,0,1,-\mathbf{1})\in\mathbb{R}^{n+2}\}_{n=1}^{\infty}$ that we are studying in this section. Thus, our choice of the sequence of reals $y$ to reach this best bound must be the one corresponding to the minus sign before the square root, i.e., $\frac{-b-\sqrt{b^2-4ac}}{2a}.$ After these computations we can now therefore write the next result.

\begin{lema}[Optimal choice]\label{optimaly}
Suppose the conditions in Equation \ref{conditionsy} are met. The optimal sequence of reals $y\colon\mathbb{N}\to\mathbb{R}$ providing the best possible bound linearizing through the sequence of vectors $v=\{(y,0,1,-\mathbf{1})\in\mathbb{R}^{n+2}\}_{n=1}^{\infty}$ is $y=\frac{-b-\sqrt{b^2-4ac}}{2a}$ with $N'D-ND':=ay^2+by+c.$
\end{lema}

\begin{proof}
The arguments in the computation above shows that this is the root corresponding to the value that provides the optimal bound of this form. The proof stems directly from that argument.
\end{proof}

Now we need to greatly extreme our care. The reason for this is the appearance of square roots while we try to estimate asymptotic growths. We explain the problems with this in greater detail in the next section. For more information on the problems of dealing with radicals beyond our scope here, see, e.g., \cite{caviness1976simplification,zippel1985simplification,falkensteiner2024transforming}.

\section{Correct management of radicals using conjugation}\label{correct}

We are forced to begin this section with a warning. Square roots can be problematic if we do not deal with them with enough care.

\begin{warning}[Freeing radicals]
Our expressions will start to show radicals from now on because of the expression we have just computed for the optimal $y$. Doing asymptotics with radicals is problematic and cannot in general be done directly when the dominant term inside the radical and the dominant term outside cancel out. For this reason, we have to deal with conjugates when this happens. Obtaining a $0$ for dominance in asymptotic growth is usually not acceptable as just writing this kills all the growth hidden behind the annihilating dominant terms. \textbf{The only time} when obtaining a $0$ for a difference in asymptotic algebra is justified is when the two expressions are \textit{exactly the same} so we obtain a \textit{real} $0$ as the end expression. Thus, from now on, dealing with expressions of the form $a-\sqrt{b}$ will be a task that will require additional work.
\end{warning}

There was a condition that we should check for $y$, as showed in Equation \ref{conditionsy}. In the spirit of our warning above, we show the next relevant example of the phenomena we explained there. The expression in the next example comes from the computations above and it is therefore stemming from our analysis of the asymptotics of the bounds we are studying.

\begin{ejemplo}[Dominant root growths annihilate]
Observe the expression of $N$ of Lemma \ref{formofboundfirstwinner} after we substitute for the optimal $y$ we computed in Lemma \ref{optimaly}. We can see that the terms winning in the numerator (inside and outside of the root) annihilate as the expression obtained if we only look at dominant terms is \begin{gather*}\frac{1}{7}2^{4 + 3n}n + \frac{1}{7}2^{5 + 3n}n + \frac{1}{7}2^{6 + 3n}n
-\sqrt{\frac{1}{49}2^{8 + 6n}n^2 + \frac{1}{49}2^{12 + 6n}n^2 + \frac{1}{49}2^{13 + 6n}n^2}\\=2^{4+3n}n-\sqrt{2^{8+6n}n^2}=0.\end{gather*} This is precisely the problem we are talking about.
\end{ejemplo}

Therefore, in order to really determine the growth of the numerator we have to proceed differently. We have to follow the next recipe.

\begin{procedimiento}[Multiplication by conjugate, difference of squares and division]\label{processy}
In particular, here we see that the conjugate does not annihilate in dominant terms as it gives $2^{4+3n}n+\sqrt{2^{8+6n}n^2}=2^{5+3n}n$ and therefore we can multiply by it to obtain that the difference of squares is dominantly $2^{4 n+8} 3^{n+1} n$ so our numerator is dominantly $\frac{2^{4 n+8} 3^{n+1} n}{2^{5+3n}n}=3^{n+1}2^{n+3}.$ For the denominator we proceed in the usual way and therefore we get that $y\sim$ \begin{gather*}\frac{3^{n+1}2^{n+3}}{-\frac{1}{3}2^{4 + 2n}n + \frac{1}{3}2^{6 + 2n}n
}=\frac{3^{n+1}2^{n+3}}{2^{4+2n}n}=\frac{3^{n+1}}{2^{n+1}n}.\end{gather*}
\end{procedimiento}

With this, in particular, we can now ensure that the sequence $y\colon\mathbb{N}\to\mathbb{R}$ we computed in the section before verifies our previous restrictions. Indeed, an easy computation makes us sure of the fact that this sequence $y\colon\mathbb{N}\to\mathbb{R}$ verifies the conditions in Equation \ref{conditionsy}.

\begin{lema}[The $y$ we wanted]
The sequence $y\colon\mathbb{N}\to\mathbb{R}$ we determined in Lemma \ref{optimaly} verifies the conditions in Equation \ref{conditionsy}.
\end{lema}

\begin{proof}
As we computed in Procedure \ref{processy} above that $y\sim\frac{3^{n+1}}{2^{n+1}n}$, it is clear that $y>1$ and that $\lim_{n\to\infty}\frac{\max\{2^{n+1}y^2,2y4^{n+1}\}}{8^{n+1}}=0$ because $2^{n+1}y^2\sim\frac{9^{n+1}}{2^{n+1}n^2}$ and $\frac{9}{2}<8$ and $2y4^{n+1}\sim2\frac{6^{n+1}}{n}$ and $6<8$. This completes the proof of the adequacy of the sequence $y\colon\mathbb{N}\to\mathbb{R}$ we chose.
\end{proof}

This proves that $y$ chosen in this way verifies our growth constraints. Hence it works as a choice for providing a \textit{correct} bound. Now, finally, we only have to check that this choice of $y$ provides a multivariate bound \textit{effectively} improving the optimal univariate bound obtained above. Substituting $y$ for this computed sequence in $\frac{N}{D}$ gives the improved bound we were pursuing here so the only thing that we still need to do is comparing this new bound $\mult_{v}(n)$ with the previous bound $\un(n).$

\begin{lema}[Form of the comparison]
Writing $y=\frac{f-\sqrt{g}}{h}$ and $\un=\frac{p+r\sqrt{q}}{s}$ we can express the difference $\mult_{v}(n)-\un(n)=$ \begin{gather*}
\frac{k+v+u+w}{s(\gamma+\delta\sqrt{g})}\end{gather*} with $\mult_{v}:=\frac{\alpha+\beta\sqrt{g}}{\gamma+\delta\sqrt{g}}$ and $k:=s\alpha-p\gamma,v:=(s\beta-p\delta)\sqrt{g},u:=-r\gamma\sqrt{q}$ and $w:=-r\delta\sqrt{gq}$.
\end{lema}

\begin{proof}
For this we compute $\mult_{v}(n)-\un(n)$. We do it step by step for maximum clarity. First, note that $y^{2}=\frac{f^{2}+g-2f\sqrt{g}}{h^{2}}$ and we can substitute these values in the expression of $\mult_{v}:=\frac{N}{D}$ and kill denominators inside $N$ and $D$ writing $\mult_{v}=\frac{N(y)}{D(y)}=\frac{h^{2}N(y)}{h^{2}D(y)}$ so at the end we obtain the manageable expression $\mult_{v}=\frac{\alpha+\beta\sqrt{g}}{\gamma+\delta\sqrt{g}}.$ Finally, remember that we wrote $\un=\frac{p+r\sqrt{q}}{s}$ and, therefore, in order to study the difference between these two bounds, the most convenient expression is $\mult_{v}-\un=$\begin{gather*}
\frac{s(\alpha+\beta\sqrt{g})-(p+r\sqrt{q})(\gamma+\delta\sqrt{g})}{s(\gamma+\delta\sqrt{g})}=\\\frac{(s\alpha-p\gamma)+(s\beta-p\delta)\sqrt{g}-r\gamma\sqrt{q}-r\delta\sqrt{gq}}{s(\gamma+\delta\sqrt{g})}=\\\frac{k+v+u+w}{s(\gamma+\delta\sqrt{g})}\end{gather*} with $k=s\alpha-p\gamma,v=(s\beta-p\delta)\sqrt{g},u=-r\gamma\sqrt{q}$ and $w=-r\delta\sqrt{gq}$, as we wanted to show.
\end{proof}

We want to be able to compute correctly the dominant behaviour of both the numerator and the denominator. The main problem with these expressions is that they contain square roots. Hence we need to deal with these carefully. We cannot directly try to proceed with dominant terms due to the many cancellations happening both in the denominator and in the numerator. We see in detail the problems produced by these cancellations. Simple computations show the following.

\begin{hecho}[Annihilation of dominant terms]
The dominant terms of the summands in our decomposition of the numerator are \begin{gather*}
k\sim -2^{11 n+15} 3^{n+1} n^{4},\\
v\sim(2^{8n+11} 3^{n+1} n^{3})\sqrt{2^{6 n+8} n^{2}}=2^{11n+15} 3^{n+1} n^{4},\\
u\sim-(2^{2 n+3} 3^{n+1})(2^{6n+9} n^{3})\sqrt{2^{6n+6} n^{2}}=-2^{11n+15}3^{n+1}n^{4},\\
w\sim-(2^{2 n+3} 3^{n+1})(-2^{3 n+5}n^{2})\sqrt{(2^{6 n+8} n^{2})(2^{6n+6} n^{2})}=2^{11n+15}3^{n+1}n^{4}.\end{gather*}
\end{hecho}

It is easy to see that these dominant terms annihilate. This is a big problem. We need to keep track of the terms surviving the big cancellation in order to be able to really establish the asymptotics of our difference. Then, in order to continue, we have to find the way to keep the largest surviving terms. We use the good behaviour of the corresponding conjugates to free some radicals via multiplication. We proceed similarly with the denominator. This is what we do next.

\begin{remark}[Procedure \ref{processy} saves us]
We continue freeing some radicals through the use of the correct conjugates. The situation in the denominator is similar. In particular, we have the luck that the conjugate expressions in the numerator \begin{gather*}(k+u)-(v+w)\sim-2^{11 n+15} 3^{n+1} n^{4}-2^{11 n+15} 3^{n+1} n^{4}\\-(2^{11 n+15} 3^{n+1} n^{4}+2^{11 n+15} 3^{n+1} n^{4})=\\-4\cdot2^{11 n+15} 3^{n+1} n^{4}=-2^{11n+17} 3^{n+1} n^{4}\end{gather*} and in the denominator \begin{gather*}\gamma-\delta\sqrt{g}\sim 2^{6 n} n^{3}512-(-32) 2^{3 n} n^{2}\sqrt{n^{2} 2^{6 n}256}=\\2^{6 n} n^{3}512-((-32) 2^{3 n} n^{2})(n 2^{3 n}16)=2^{6 n+10} n^{3}\end{gather*} do not annihilate.
\end{remark}

Observe however that \begin{gather*}\gamma+\delta\sqrt{g}\sim 2^{6 n} n^{3}512+(-32) 2^{3 n} n^{2}\sqrt{n^{2} 2^{6 n}256}=0.\end{gather*} Now we can finally proceed with the comparison in the next section.

\section{Final comparison of the bounds}\label{finalcom}

A few computations lead us where we want. In particular, we obtain the following lemma.

\begin{lema}[Difference growth]\label{lemafinal}
We have the asymptotic growth of the difference of bounds $$\mult_{v}-\un\sim\frac{1}{2}\left(\frac{3}{4}\right)^n.$$
\end{lema}

\begin{proof}
Since managing the denominator is easy, we proceed first with the numerator. We do this multiplying by the nice conjugates saw above. This will free some radicals. We warn the reader that in the next explanation of the structure of the computations involved all the names are locally set. Since we have seen that $(k+u)-(v+w)$ is dominantly $-2^{11 n+17} 3^{n+1} n^4$, we can multiply by it in order to get a nice difference of squares. Hence, multiplying the numerator $k+u+v+w$ by $(k+u)-(v+w)$ we obtain precisely $(k+u)^2-(v+w)^2=k^2+u^2+2ku-v^2-w^2-2vw=(k^2+u^2-v^2-w^2)+(2ku)+(-2vw)=r+s+t.$ We will need a further use of conjugates to determine the real growth of this because the expressions of $s$ and $t$ contain $\sqrt{q}.$ We have to see first what is the growth of the conjugate $r-(s+t)$. In order to see this, we first need to establish the real growth of $s+t$, which is not immediate because there is a cancellation again as it is easy to see that $ku\sim (2^{11n+15} 3^{n+1} n^{4})^2\sim vw$. Thus the use of conjugates is necessary again. We can immediately compute the asymptotic growth of the part with no roots involved $r\sim-2^{19 n+31} 3^{3 n+2} n^7$. It is clear that $s-t=(2ku)-(-2vw)=2(ku+vw)\sim 4(2^{11n+15} 3^{n+1} n^{4})^2=2^{22n+32} 3^{2n+2} n^{8}$ while computing again we can see that $(s+t)(s-t)=s^2-t^2\sim -2^{63 + 41 n} 3^{4 + 5 n} n^{15}$ so we obtain therefore that $s+t=\frac{s^2-t^2}{s-t}\sim\frac{-2^{63 + 41 n} 3^{4 + 5 n} n^{15}}{2^{22n+32} 3^{2n+2} n^{8}}=-2^{31 + 19n} 3^{2 + 3n} n^{7}\sim r$. This implies that $r+s+t\sim-2^{19 n+32} 3^{3 n+2} n^{7}$ and therefore we obtain that $k+v+u+w$ is dominantly $\frac{-2^{19 n+32} 3^{3 n+2} n^{7}}{-2^{11 n+17} 3^{n+1} n^4}=2^{8 n+15} 3^{2 n+1} n^3.$ Proceeding similarly with the denominator we get that the conjugate of the problematic factor is dominantly $\gamma-\delta\sqrt{g}\sim$ \begin{gather*} 2^{6 n} n^{3}512-(-32) 2^{3 n} n^{2}\sqrt{n^{2} 2^{6 n}256}=\\2^{6 n} n^{3}512-((-32) 2^{3 n} n^{2})(n 2^{3 n}16)=2^{6 n+10} n^{3}\end{gather*} while the resulting difference of squares is dominantly $2^{12 n+20} n^{5}$ so, finally, we can see that the problematic factor is dominantly $\gamma+\delta\sqrt{g}\sim\frac{2^{12 n+20} n^{5}}{2^{6 n+10} n^{3}}=2^{6 n+10} n^2$ so, as $s\sim3^{n} 2^{4 n} n 192,$ the denominator is dominantly $$s(\gamma+\delta\sqrt{g})\sim(3^{n} 2^{4 n} n 192)(2^{6 n+10} n^2)=2^{10 n+16} 3^{n+1} n^3.$$ All in all, we see that the difference is dominantly $$0<\frac{2^{8 n+15} 3^{2 n+1} n^3}{2^{10 n+16} 3^{n+1} n^3}=\frac{3^n}{2^{1+2n}}=\frac{1}{2}\left(\frac{3}{4}\right)^n\to0^{+} \mbox{\ when\ } n\to\infty,$$ as we wanted. This finishes our proof.
\end{proof}

Observe that this implies $\mult_{v}-\un\to 0^{+} \mbox{\ when\ } n\to\infty$. We have reached our destination in this section. But some comments in the following sections will convince us that this is not the real end of the path that we have initiated here.

\section{Multivariate approach (not just slightly) beats univariate}\label{multapr}

Now we can analyze calmly and carefully the implications of the results proved through the three section above. In particular, we have proved the result we wanted while taking care of the difficulties and technicalities that we warned the reader about before we started.

\begin{reminder}[Packing up things proves Proposition \ref{multundiff}]
Lemma \ref{lemafinal} finishes the proof. The growth mentioned in the cited proposition has been finally established. For this, we had to complete several steps of freeing radicals and conjugating.
\end{reminder}

We still have a problem with the result we obtained. The difference above is small. We get an improvement, but it vanishes when $n$ goes to infinity.

\begin{remark}[Difference is small]
The fact that the sequence of differences of bounds verifies $\mult_{v}-\un\to 0^{+} \mbox{\ when\ } n\to\infty$ tells us that we are in the right path. However, at the same time, we see that we obtain here an improvement that diminishes when $n$ grows.
\end{remark}

This means that there is still further work to do because we expect better improvements. Indeed, our numerical experiments tell us that it looks like this improvement could be vastly surpassed.

\begin{observacion}
Simple numerical experiments using either directly the determinant of the relaxation or indirectly other sequences of vectors to linearize the problem make us suspect that it is possible to establish an explosive improvement of the bound obtained. However, in order to achieve a proof of this fact, it is necessary to look at more sophisticated (heuristic) ways of finding approximations to the generalized eigenvectors of the relaxation restricted to the diagonal.
\end{observacion}

In particular, further and more technical numerical experiments will be necessary to advance in the direction of understanding better the structure of the corresponding (guesses of) eigenvectors. These experiments point towards a direction promising us better bounds. Thus, exploring that direction will force us to find a more elaborated structure for the vectors we use in our linearizations. The description of that structure and the path we follow to find such structure lie beyond the scope of this article, but we will make some remarks about how we can manage to look in the right direction for these more nuanced linearizations.

\section{A look beyond for greater improvements}\label{alook}

In order to obtain these better linearizations, we will need to develop intuitions about how to build the corresponding more sophisticated sequences of vectors that we will need. We will access and create these intuitions and guesses about the form and the structure of the new sequences of vectors through numerical experiments. These experiments will reveal us more information about the structure of the matrices involved in the relaxation. This information will show helpful in our future analysis. In particular, we will intend to obtain insights about the growth of the difference between the original bound and the one provided by the relaxation. Additionally, these experiments will naturally provide hints about the structure of the linearizing vectors we are searching for.

\begin{observacion}[Possibilities and conditions for better bounds]
What we have seen in this article tells us that, to improve the bound, we have to manage to split the information about the different variables in the multivariate polynomial when we look at the relaxation. The sequence of vectors we used here $v=\{(y,0,1,-\mathbf{1})\in\mathbb{R}^{n+2}\}_{n=1}^{\infty}$ does not split the tail of variables $(x_{3},\dots,x_{n+1})$ because all of them get assigned the value $-1.$ Thus, we see that just the first variable $x_{2}$ is split from the rest in the tail $(x_{3},\dots,x_{n+1})$.
\end{observacion}

Thus we can see that we are actually just using a bivariate splitting ($x_{2}$ versus the tail of variables $(x_{3},\dots,x_{n+1})$) here. Therefore it is clear that we are not using the full power of the multivariate splitting provided by the multivariate Eulerian polynomials.

\begin{objetivo}
A task for a next paper is finding a solution to this deficiency \textbf{of the execution} of our method here. We have to work in developing this improvement of the execution of the method because it is clear that almost all potential advantages of exploiting effectively the further and deeper split of variables in the multivariate Eulerian polynomials is being lost when we use this \textit{way too simple} linearization.
\end{objetivo}

In particular, there is still a lot of room for improving that splitting. The reason for this is that there are many equal entries in the vectors of the sequence of vectors $v$ we chose above.

\begin{observacion}[Benefits of further variation in entries]
We presume that, in general, extracting more information about the relaxation using linearizations pass irremediably by charging the entries with variable (in the sense that the values will change with $n$) and different (in the sense that not so many entries repeat for the same value of $n$) sequences that can distinguish between the entries of the vector they appear in. In this way, we should be able to split the information in the LMP defining the spectrahedron of the relaxation through the linearization using the sequence of vectors thus constructed.
\end{observacion}

The eigenvectors themselves might have expressions too complicated and wild to deal with. Therefore, the real challenge will be to find approximated guesses for eigenvectors having expressions that are nice to deal with in computational, numeric and symbolic terms.

\begin{remark}
Observe that there is only one different entry (having value $1$) among the significative (the $0$ is not relevant as it corresponds to the ghost variable $x_{1}$) tail of the vectors in the sequence of vectors $v$ above. Setting a greater proportion of entries different will allow us to distinguish and therefore pour more information about the relaxation into our linearizations.
\end{remark}

This remark tells us where lies the problem for extracting more information about the relaxation during the process of linearizing. In particular, now we understand that a successful way to do this is the next process.

\begin{procedimiento}[Searching for better bounds through higher multivariability]\label{procedurebeating}
In order to accomplish the objective of linearizing extracting more information from the relaxation, it seems heuristically clear that the number of different entries will have to grow with $n$. This is so because such growth requirement will allow us to accommodate more and better quality information about the relaxation in the corresponding linearization through the sequence of vectors thus constructed. At the same time, it seems reasonable to avoid using static sequences. These entries having static sequences will have to change in order to resemble the entries of the actual eigenvectors, which are expected to change when the index $n$ does. That is, we see that we need to construct entries that vary when $n$ varies.
\end{procedimiento}

In regard of the procedure described above, observe that, in the simple sequence of vectors used here $v$, just the sequence of reals $y$ occupying the first entry varies because we chose it that way through an optimization process. In defense of the convenience of the procedure described here, note then that, in particular, that slight variation in the sequence of vectors was already enough to beat the application of the relaxation to the univariate polynomial. We need to follow suit and make also a greater proportion of the entries variable depending on $n$.

But all this is goes beyond the scope of this paper. Here we just wanted to introduce the relaxation as a way to globally bound Eulerian RCSs globally through a method that, when we measure its accuracy in the diagonal, already shows us how close it lies to the actual RCS near that diagonal. This fact opens the door to future explorations in the applicability of this method in many different directions.

We do not mean just to stay in this example and refine it at the maximum, as we commented in this section. Contrary to that, we actually seek to find new venues and properties where this method can shed light to globally approximate RCSs of other (structured) combinatorial polynomials. We refer here to polynomials like these in \cite{branden2016multivariateviso,branden2016multivariate,zhang2019multivariate}.

Another related question comes along the lines of establishing more sophisticated methods to measure how good these global approximations (relaxations) are. We need to think further about these more elaborated measures because we saw here that just looking at the diagonal, although a nice first step to see how promising the technique is, seems insufficient to determine and measure the real accuracy of the obtained approximations far from the corresponding diagonals (or, in more generality, lines where the corresponding univariate polynomials are injected into the multivariate liftings). Some related concepts of measuring accuracy of global approximations to zero sets of multivariate polynomials can be consulted in \cite{boese2003accurate, henrion2023polynomial, melczer2021multivariate, pemantle2008twenty, mourrain2000multivariate}.

We need to explore these venues of problems further if we want to obtain more satisfactory answers to our questions about this general method. For this reason, these necessary further explorations are briefly discussed in the next conclusion.

\section{Conclusion}\label{con}

We will put Procedure \ref{procedurebeating} in practice in the future. This will set in motion the machinery that will allow us to show that the application of the relaxation to the multivariate Eulerian polynomials produces indeed a much greater improvement in the (asymptotic growth) quality of our bounds. This growth is indeed explosive because it is exponential.

\begin{objetivo}[Seeking for explosive exponential differences in deeper order asymptotics]
After we find a structure for a vector that seems to numerically accomplish the improvement we are looking for, we will compare it to $\un$. This comparison, as before, will be done through a computation of a difference. This difference will therefore give us the point of the exponential asymptotic scale we are using at which such improvement happen. We will be happy if this point is a point in the exponential scale $a^{n}$ with $a>1$. This will tell us that the new bound explosively (exponentially) diverges from the bounding sequence $\un$ obtained here when $n$ grows. This is in fact the objective we pursue with the further development of our method exposed in Procedure \ref{procedurebeating}.
\end{objetivo}

We are actually able to accomplish the objective above after some experimentation. The journey towards the discovery of the structure of the approximated generalized eigenvector we use and the proof of the fact that this structure does indeed produce the improvement we seek is too long to fit in this paper, whose scope is fundamentally different. We therefore defer these developments of the method here described to future publications on these improvements and refinements. For now and to close the discussion on this paper here, we can confirm in this conclusion that we managed to explosively win against the bounding sequence $\un$ with a sequence of vectors having the characteristics described in Procedure \ref{procedurebeating}. In particular and as we wanted, this sequence of vectors produces a linearization giving a bound whose sequence of differences with the bounding sequence $\un$ grows exponentially with $n$.

\section*{Declaration of competing interest}

The author declare that he has no known competing financial interests or personal relationships that could have appeared to influence the work reported in this paper.

\section*{Acknowledgements}

I thank my doctoral advisor Markus Schweighofer for his support during the development of this research. This work has been generously supported by European Union’s Horizon 2020 research and innovation programme under the Marie Skłodowska Curie Actions, grant agreement 813211 (POEMA).

\end{document}